\documentclass[a4paper, portrait]{amsart}
\usepackage{amsmath}
\usepackage{graphicx}
\usepackage{hyperref}
\usepackage{tabularx}
\usepackage{float} 

\makeatletter
\newcommand{\addresseshere}{%
	\enddoc@text\let\enddoc@text\relax
}
\makeatother

\usepackage{amsmath, amssymb, graphicx}
\usepackage[normalsize]{caption}
\usepackage{subcaption}

\newcommand{\beq}{\begin{equation}}\newcommand{\eeq}{\end{equation}}
\newcommand{\ba}{\begin{array}}
\newcommand{\ea}{\end{array}}
\newcommand{\bea}{\begin{eqnarray}}
\newcommand{\eea}{\end{eqnarray}}
\newtheorem{thm}{Theorem}[section]
\newtheorem{cor}[thm]{Corollary}
\newtheorem{lem}[thm]{Lemma}
\newtheorem{prop}[thm]{Proposition}

\newtheorem{defn}[thm]{Definition}
\newtheorem{rem}[thm]{Remark}
\newtheorem{assume}[thm]{ASSUMPTION}

\newcommand {\RR} {\mathbf{R}}

\newcommand{\set}[1]{\left\{#1\right\}}

\newcommand {\parx}{\frac{d}{dx}}
\renewcommand {\part}{\frac{\partial}{\partial t}}
\newcommand{\suml}{\sum\limits}

\newcommand{\fb}{\mathfrak{b}}
\newcommand{\fu}{\mathfrak{u}}
\newcommand{\fv}{\mathfrak{v}}
\newcommand{\fr}{\mathfrak{r}}
\newcommand{\fs}{\mathfrak{s}}
\newcommand{\fw}{\mathfrak{w}}
\newcommand{\fz}{\mathfrak{z}}
\newcommand{\fe}{\mathfrak{e}}
\newcommand{\fg}{\mathfrak{g}}

\newcommand{\fn}{\mathfrak{n}}
\newcommand{\fU}{\mathfrak{U}}

\numberwithin{equation}{section}

\title{Finite Difference Approach to Fourth-Order Linear Boundary-Value Problems}

\author{Matania Ben-Artzi}

\address{Matania Ben-Artzi: Institute of Mathematics, The Hebrew University, Jerusalem 91904, Israel}

\email{mbartzi@math.huji.ac.il}

\author{Benjamin Kramer}

\address{Benjamin Kramer: Department of Computer Science, The Hebrew University, Jerusalem 91904, Israel}

\email{benjamin.kramer@mail.huji.ac.il}

\thanks{It is a pleasure to thank Jean-Pierre Croisille and Dalia Fishelov for many helpful discussions. The scheme presented here, as well as the numerical examples, were included in the M.Sc. thesis of B. Kramer, supervised by M. Ben-Artzi and N. Linial.}


\date{\today}

\begin{document}

	\begin{abstract}

		Discrete approximations to the equation
		\begin{equation*}
			L_{cont}u = u^{(4)} + D(x) u^{(3)} + A(x) u^{(2)} + (A'(x)+H(x)) u^{(1)} + B(x) u = f, \; x\in[0,1]
		\end{equation*}
		are considered. This is an extension of the Sturm-Liouville case $D(x)\equiv H(x)\equiv 0$ ~\cite{IMA_SL_paper} to the non-self-adjoint setting.

		The ``natural'' boundary conditions in the Sturm-Liouville case are the values of the function and its derivative. The inclusion of a third-order discrete derivative entails a revision of the underlying discrete functional calculus. This revision forces evaluations of accurate discrete approximations to the boundary values of the second, third and fourth order derivatives. The resulting functional calculus provides the discrete analogs of the fundamental Sobolev properties--compactness and coercivity. It allows to obtain a general convergence theorem of the discrete approximations to the exact solution.

		Some representative numerical examples are presented.

	\end{abstract}

	\maketitle

	\section{\textbf{INTRODUCTION}}
		A discrete functional calculus was developed in ~\cite{IMA_SL_paper}, involving derivatives up to fourth-order and establishing full analogs of the fundamental Sobolev properties (in particular, coercivity and compactness). It was applied to obtain an optimal convergence rate for numerical simulation of fourth-order Sturm-Liouville problems on an interval.

		In this article, we present an extension of the discrete calculus that is applicable for the numerical solution of general (namely, non- self-adjoint) fourth-order linear boundary-value problems.

		This extension is, somewhat surprisingly, non trivial. In fact, the main difficulty rests with the third-order derivative, that does not appear in the Sturm-Liouville case. In turn, its inclusion forces a revision of the entire discrete calculus, and in particular a high-order evaluation of approximate values of the second, third and fourth-order derivatives of the solution on the boundary. This should be compared to the Sturm-Liouville case, where the full construction relied solely on the ``natural'' boundary conditions, namely, the given boundary values of the function and its first order derivative.

		While the third-order derivative does not appear in self-adjoint problems, it may play an important role (for example, the convection term in the streamfunction formulation of the Navier-Stokes equations ~\cite{Navier_Stokes}). The third-order derivative term causes the differential operator to be non-self-adjoint, and complicates the analytical approach beyond the discrete elliptic theory that was introduced in \cite{IMA_SL_paper}.

		This work, as well as previous ones, relies on the notion of Hermitian derivatives, that was first introduced in the context of numerical solution of ordinary differential equations ~\cite[Chap. IV]{Collatz}). The first-order Hermitian derivative was used as the building block in the construction of discrete harmonic and biharmonic operators ~\cite[Chapter 10]{book}. The close relation of the discrete biharmonic operator to cubic splines was studied in \cite{Splines}. The high-order discrete derivatives, based on the first-order Hermitian derivative, are compact finite-difference operators.
		
		\bigskip
		We consider here discrete approximation to the equation
		\begin{align} \label{utxxxxcontgen}
			L_{cont}u = & \Big(\parx\Big)^4 u + D(x)\Big(\parx\Big)^3 u + A(x)\Big(\parx\Big)^2u \\ \nonumber + & (A'(x)+H(x))\Big(\parx\Big)u+B(x)u = f,\quad
			x\in\Omega=[0,1],
		\end{align}
		where $A(x), B(x), D(x) $ and $ H(x)$ are real functions, $A(x)\in C^1(\Omega)$ and \\
		$B(x), D(x), H(x)\in C(\Omega)$.
		The equation is supplemented with homogeneous boundary conditions
		\begin{equation} \label{eqbdrydata}
			u(0)=\parx u(0)=u(1)=\parx u(1)=0.
		\end{equation}

		Note that non-homogeneous boundary conditions are accommodated by a modification of the right-hand side function $f(x)$.

		The \textbf{Sturm-Liouville} case, namely, $D(x) \equiv H(x) \equiv 0$, was treated in detail in ~\cite{IMA_SL_paper}.
		For the general form $L_{cont}$ it is well-known ~\cite{Naimark_book} that the spectrum consists of a sequence of (not necessarily real) eigenvalues whose absolute values diverge to infinity. We impose the following assumption (that can always be validated by a small shift of $B(x)$),
		\begin{assume} \label{assumespLcont}
			Zero is not an eigenvalue of $L_{cont}.$
		\end{assume}

		The paper is organized as follows.
		
		In Section ~\ref{secdiscsetup} we recall the functional calculus presented in ~\cite{IMA_SL_paper} and extend it to the non-self-adjoint case. In particular the discrete third-order derivative is introduced.

		As mentioned above, the introduction of the third-order derivative cannot be accomplished without obtaining an adequate definition of discrete high-order derivatives at boundary and near-boundary grid points. Evaluating these derivatives can be regarded as an implementation of ``Dirichlet to Neumann'' operators, involving the natural data (i.e., boundary values of the unknown function and its derivative), coupled with functional values at interior points.
		
		Section ~\ref{secboundary} is devoted exactly to this task. It serves as the main technical tool in this paper. We recall that the reliance of high-order schemes on discrete accurate approximation of boundary values of higher derivatives is common; a well-known case is the need for boundary values of the vorticity when dealing with streamfunction-vorticity formulation of the Navier-Stokes equations ~\cite[Section 2]{WE}.

		Section ~\ref{seccompact} is concerned with supplementary material about the discrete functional calculus (developed in ~\cite{IMA_SL_paper}). More specifically, Corollary ~\ref{cordelta3conv} deals with an appropriate extension of the ``discrete Rellich theorem'', namely the convergence of the sequence of discrete third-order derivatives, that are absent in the self-adjoint case.

		The main object of this paper is the discrete equation ~\eqref{utxxxxdiscgen} approximating Equation ~\eqref{utxxxxcontgen}. It is presented in Section ~\ref{secdiscreteconvergence}. Building on the discrete functional properties, a general convergence theorem (Theorem ~\ref{thmconvgeneralscheme}) is established .

		In Section ~\ref{secnumerresults} we present and discuss two numerical test problems.

		\subsection{Related work}
			While there exists a vast amount of works on the numerical treatment of second-order boundary value problems, the literature on higher-order problems is considerably less extensive. A rigorous treatment of iterative methods, applicable also to nonlinear problems, is expounded in ~\cite{agarwal}. A numerical treatment of such methods is presented in ~\cite{xu}. In the framework of finite-element methods, we mention spline methods ~\cite{spline1,spline2} and Galerkin methods ~\cite{Galerkin2005,Galerkin2010}. A recent work ~\cite{RokhlinSecondKindIntegral} has addressed the same problem as the one studied here by means of finite difference methods. It uses second-kind integral equations and an iterative correction method to achieve optimal accuracy.

	\section{\textbf{DISCRETE FUNCTIONAL CALCULUS}} \label{secdiscsetup}

		We equip the interval $\Omega=[0,1]$ with a uniform grid
		\begin{equation*}
			x_j=jh,\quad 0\leq j\leq N,\quad h=\frac{1}{N}.
		\end{equation*}
		
		The approximation is carried out by grid functions $\fv$ defined on $\set{x_j,\,0\leq j\leq N}.$ The space of these grid functions is denoted by $l^2_h.$
		For their components we use either $\fv_j$ or $\fv(x_j).$

		For every smooth function $f(x)$ we define its associated grid function
		\begin{equation} \label{eqdefustar}
			f^\ast_j=f(x_j),\quad 0\leq j\leq N.
		\end{equation}
		The discrete $l^2_h$ scalar product is defined by
		\begin{equation*}
			(\fv,\fw)_h=h\suml_{j=0}^N\fv_j\fw_j,
		\end{equation*}
		and the corresponding norm is
		\begin{equation} \label{eqdefl2norm}
			|\fv|_h^2=h\suml_{j=0}^N\fv_j^2.
		\end{equation}

		For linear operators $\mathcal{A}:l^2_h\to l^2_h$ we use $|\mathcal{A}|_h$ to denote the operator norm.	
		The discrete sup-norm is
		\begin{equation} \label{eqdeflinfnorm}
			|\fv|_\infty=\max_{0\leq j\leq N}\set{|\fv_j|}.
		\end{equation}

		The discrete homogeneous space of grid functions is defined by
		\begin{equation*}
			l^2_{h,0}=\set{\fv,\,\,\fv_0=\fv_N=0}.
		\end{equation*}

		Given $\fv\in l^2_{h}$ we introduce the basic (central) finite difference operators
		\begin{gather}
			\label{eq_delta_op}
			(\delta_x\fv)_j=\frac{1}{2h}(\fv_{j+1}-\fv_{j-1}),\quad 1\leq j\leq N-1,\\
			\label{eq_delta2_op}
			(\delta^2_x\fv)_j=\frac{1}{h^2}(\fv_{j+1}-2\fv_j+\fv_{j-1}),\quad 1\leq j\leq N-1,
		\end{gather}

		The cornerstone of our approach to finite difference operators is the introduction of the \textbf{Hermitian derivative} of $\fv\in l^2_{h,0},$ that will replace $\delta_x.$ It will serve not only in approximating (to fourth-order of accuracy) first-order derivatives, but also as a fundamental building block in the construction of finite difference approximations to higher-order derivatives.

		First, we introduce the ``Simpson operator''
		\begin{equation} \label{eqdefsigmasimpson}
			(\sigma_x\fv)_j=\frac16\fv_{j-1}+\frac23\fv_{j}+\frac16\fv_{j+1},\quad 1\leq j\leq N-1.
		\end{equation}

		Note the operator relation (valid in $l^2_{h,0}$)	
		\begin{equation}
			\sigma_x=I+\frac{h^2}{6}\delta_x^2,
		\end{equation}
		so that $\sigma_x$ is an ``approximation to identity'' in the following sense.

		Let $\psi\in C^\infty_0(\Omega),$ then
		\begin{equation} \label{eqsigmaapproxj}
			|(\sigma_x-I)\psi^\ast|_{\infty}\leq C h^2\|\psi''\|_{L^\infty(\Omega)},
		\end{equation}
		which yields
		\begin{equation} \label{eqsigmaapproxh}
			|(\sigma_x-I)\psi^\ast|_h\leq C h^2\|\psi''\|_{L^\infty(\Omega)}.
		\end{equation}
		In the above estimates the constant $C>0$ is independent of $h$, $\psi$.

		The Hermitian derivative $\fv_x$ is now defined by
		\begin{equation} \label{eqdefhermit}
			(\sigma_x\fv_x)_j=(\delta_x\fv)_j,\quad 1\leq j\leq N-1.
		\end{equation}

		In analogy with the operator notation $\delta_x\fv$ we shall find it convenient to use the operator notation
		\begin{equation}\label{eqoptilddelta}
			\widetilde{\delta_x}\fv=\fv_x.			
		\end{equation}

		The truncation error for $\widetilde{\delta_x}$ (see \eqref{eq_her_der_trunc_error}) is
		\begin{equation} \label{eqtrunc1st}
			(\widetilde{\delta_x} u^\ast)_j = u^{(1)}(x_j) + O(h^4), \quad \forall j\in\{1,...,N-1\}.
		\end{equation}

		\begin{rem} \label{remzerodiscderiv}
			In the definition ~\eqref{eqdefhermit}, the values of $(\fv_x)_j,\,\,j=0,N,$ need to be provided, in order to make sense of the left-hand side (for $j=1,N-1$). If not otherwise specified, \textit{we shall henceforth assume that, in accordance with the boundary condition ~\eqref{eqbdrydata}}, $\fv_x\in l^2_{h,0},$ namely
			$$(\fv_x)_0=(\fv_x)_{N}=0.$$

			\textit{In particular, the linear correspondence $ l^2_{h,0}\ni \fv\to \fv_x\in l^2_{h,0}$ is well defined, but not onto, since $\delta_x$ has a non-trivial kernel.}
		\end{rem}

		Based on the Hermitian derivative, discrete versions of the second, third and fourth-order derivatives are introduced as follows. The precise estimates of the truncation errors are stated in Subsection ~\ref{sub_section_error_analysis} below.

		\begin{itemize}
			\item A higher-order replacement to the operator $\delta^2_x$ (see ~\cite[Equation (15)]{optimal-paper}, ~\cite[Equation (10.50)(c)]{book}) is defined by
			\begin{equation} \label{eqdefdelta2tild}
				(\widetilde{\delta_x^2}\fv)_j=2(\delta_x^2\fv)_j-(\delta_x\fv_x)_j,\quad 1\leq j\leq N-1.
			\end{equation}

			Note that, in accordance with Remark ~\ref{remzerodiscderiv} the operator $\widetilde{\delta}_x^2$ is defined on grid functions $\fv\in l^2_{h,0},$ such that also $\fv_x\in l^2_{h,0}.$

			\item The discrete third-order operator is defined by
			\begin{equation} \label{eqdef3rd}
				(\delta_x^3\fv)_j=2(\delta_x^2\fv_x)_j-(\delta_x\widetilde{\delta_x^2}\fv)_j, \quad 1\leq j\leq N-1.
			\end{equation}

			Clearly, to make sense of this definition at near-boundary points $j=1,N-1,$ the boundary values $(\widetilde{\delta_x^2}\fv)_0$ and $(\widetilde{\delta_x^2}\fv)_N$ must be given. They are derived by a (coupled) calculation that is detailed below in Section ~\ref{secboundary}.

			\item The biharmonic discrete operator is given by (for $\fv,\,\fv_x\in l^2_{h,0}),$
			\begin{equation} \label{eq_delta4_op}
				(\delta_x^4\fv)_j=\frac{12}{h^2}[\delta_x\fv_x-\delta_x^2\fv]_j,\quad 1\leq j\leq N-1.
			\end{equation}

			Note that the connection between the two difference operators for the second-order derivative is given by
			\begin{equation} \label{twosecondorder}
				-\widetilde{\delta_x}^2=-\delta_x^2 +\frac{h^2}{12} \delta_x^4.
			\end{equation}

		\end{itemize}

		Remark also that the operators $\sigma_x$, $\delta_x$ and $\delta_x^2$ all commute at interior points. That is, for any $\mathfrak{v} \in l_h^2$
		\begin{gather}
			\label{eq_sigma_dx_comutator}
			((\sigma_x \delta_x - \delta_x \sigma_x) \mathfrak{v})_j = 0,\\
			\label{eq_sigma_dx2_comutator}
			((\sigma_x \delta_x^2 - \delta_x^2 \sigma_x)\mathfrak{v})_j = 0,\\
			\label{eq_dx_dx2_comutator}
			((\delta_x \delta_x^2 - \delta_x^2 \delta_x) \mathfrak{v})_j = 0, \\
			\nonumber \forall j \in \{2, ..., N-2\}.
		\end{gather}

		\begin{rem} \label{remhdependence}
			Clearly the operators $\delta_x,\, \delta_x^2,\,\delta_x^3,\, \delta_x^4$ depend on $h,$ but for notational simplicity this dependence is not explicitly indicated.
		\end{rem}

		\subsection{Uniform boundedness of the discrete operators}

			The fact that the biharmonic discrete operator $\delta_x^4$ is positive (in particular symmetric) is proved in ~\cite[Lemmas 10.9,\,10.10]{book}. Therefore its inverse $\Big(\delta^4_x\Big)^{-1}$ is also positive.

			A fundamental tool (analogous to classical elliptic theory) is the coercivity property (with $C>0$ independent of $\,h$) ~\cite[Propositions 10.11,10.13]{book},
			\begin{equation} \label{eqcoercivity}
				(\delta^4_x\fz,\fz)_h\geq C\Big(|\fz|^2_h+|\delta_x^2\fz|^2_h+|\delta_x\fz_x|^2_h\Big),
			\end{equation}

			valid for any grid function $\fz\in l^2_{h,0}$ such that also $\fz_x\in l^2_{h,0}.$

			\fbox{\parbox[c]{300pt}{\textbf{ALL OPERATORS HERE AND BELOW ARE CONSIDERED AS ACTING ON GRID FUNCTIONS $\fz\in l^2_{h,0}$ SO THAT ALSO $\fz_x\in l^2_{h,0}.$}}}

			We first show that, in `` operator sense'', the second-order operator $\delta_x^2$ is comparable (independently of $h>0$) to $(\delta^4_x)^{\frac12}.$

			\begin{lem}\label{lembounddel2del4sqrt}
				The operators
				$\Big(\delta^4_x\Big)^{-\frac12}\delta_x^2$ and $\delta_x^2\Big(\delta^4_x\Big)^{-\frac12}$
				are bounded in $l^2_{h,0},$ with bounds that are independent of $h.$
			\end{lem}

			\begin{proof}
				We use the coercivity property ~\eqref{eqcoercivity} with $\fz=\Big(\delta^4_x\Big)^{-\frac12}\fw,$ and obtain
				\begin{equation} \label{del4minus12}
					\Big(\Big(\delta^4_x\Big)^{\frac12}\fw,\Big(\delta^4_x\Big)^{-\frac12}\fw\Big)_h\geq C\Big|\delta_x^2\Big(\delta^4_x\Big)^{-\frac12}\fw\Big|^2_h.
				\end{equation}

				The operator $\delta_x^2\Big(\delta^4_x\Big)^{-\frac12}$ is therefore bounded, with a bound that is independent of $h.$ The same is true (with the same bound, by a well-known fact about norms of adjoints) for its adjoint, namely,
				$\Big(\delta^4_x\Big)^{-\frac12}\delta_x^2.$
			\end{proof}

			In the sequel we shall find it useful to use slightly different (and in fact weaker) boundedness facts (again uniform with respect to $h$), that are listed in the following proposition .

			\begin{prop}\label{propbounded}
				The operators $(\delta^4_x)^{-1},\,\,\Big(\delta^4_x\Big)^{-1}\widetilde{\delta_x^2}$ and $\widetilde{\delta}_x^2\Big(\delta^4_x\Big)^{-1}$
				are bounded in $l^2_{h,0},$ with bounds that are independent of $h.$
			\end{prop}

			\begin{proof}
				The boundedness of $\Big(\delta^4_x\Big)^{-1}$ follows directly from the coercivity property ~\eqref{eqcoercivity}, by an obvious application of the Cauchy-Schwarz inequality.

				In view of ~\eqref{twosecondorder},
				\begin{equation} \label{simplify} \aligned
					\Big(\delta^4_x\Big)^{-1}\widetilde{\delta}_x^2
					=
					\Big(\delta^4_x\Big)^{-1}\delta_x^2 - \frac{h^2}{12}\Big(\delta^4_x\Big)^{-1}\delta_x^4\\=
					\Big(\delta^4_x\Big)^{-1}\delta_x^2 - \frac{h^2}{12}I.\endaligned
				\end{equation}

				It therefore suffices to prove the boundedness of $\Big(\delta^4_x\Big)^{-1}\delta_x^2.$ But this simply follows from Lemma ~\ref{lembounddel2del4sqrt} and
				\begin{equation*}
					\Big(\delta^4_x\Big)^{-1}\delta_x^2=\Big(\delta^4_x\Big)^{-\frac12}\Big(\delta^4_x\Big)^{-\frac12}\delta_x^2.
				\end{equation*}
			\end{proof}

	\section{\textbf{BOUNDARY and NEAR-BOUNDARY APPROXIMATIONS OF HIGHER DERIVATIVES}}\label{secboundary}

		In this section we develop discrete approximations for higher-order derivatives at boundary points as well as near boundary points. Such expressions are needed for the discrete approximation of the general Equation \eqref{utxxxxcontgen}. For the sake of simplicity we assume as above homogeneous boundary conditions (of the unknown function and its first derivative). Certainly general boundary conditions can be reduced to the homogeneous ones (with appropriate change of the right-hand side).

		The basic idea is to express the boundary values of the higher-order discrete derivatives $\mathfrak{u}^{(2)}_i, \mathfrak{u}^{(3)}_i, \mathfrak{u}^{(4)}_i,\,\,i=0,N$ in terms of $\mathfrak{u}$ and its Hermitian derivative.

		We begin by expanding the second-order discrete derivative operator from Equation \eqref{eqdefdelta2tild} also to the boundary points $j=0$ and $j=N$. With the values of $\mathfrak{u}_0^{(2)}, \mathfrak{u}_N^{(2)}$ to be determined below, we can write a new second-order discrete derivative operator $\widetilde{\delta_x^2}$, for any $\mathfrak{u}$ such that $\mathfrak{u}, \mathfrak{u}_x \in l^2_{h,0}$:
		\begin{gather} \label{eq_delta2tilda_op_new}
			(\widetilde{\delta_x^2} \mathfrak{u})_j = 2(\delta_x^2 \mathfrak{u})_j - (\delta_x \widetilde{\delta_x} \mathfrak{u})_j, \quad \forall j\in\{1,...,N-1\}, \\ \nonumber
			(\widetilde{\delta_x^2} \mathfrak{u})_0 = \mathfrak{u}_0^{(2)}, \quad (\widetilde{\delta_x^2} \mathfrak{u})_N = \mathfrak{u}_N^{(2)}.
		\end{gather}

		Next, with the boundary values $\mathfrak{u}_0^{(3)}, \mathfrak{u}_N^{(3)}$ to be determined below, the new operator $\widetilde{\delta_x^2}$ allows us to define a third-order discrete derivative operator $\delta_x^3$, for any $\mathfrak{u}$ such that $\mathfrak{u}, \mathfrak{u}_x \in l^2_{h,0}$. Unlike the second and fourth-order discrete derivative operators, that were derived from polynomial interpolation in \cite{IMA_SL_paper}, this expression is defined as a linear combination of two simple third-order discrete derivative approximations, as follows.
		\begin{gather} \label{eq_delta3_op}
			(\delta_x^3 \mathfrak{u})_j =
			(2\delta_x^2 \widetilde{\delta_x} \mathfrak{u})_j - (\delta_x \widetilde{\delta_x^2} \mathfrak{u})_j, \quad \forall j\in\{1,...,N-1\}, \\ \nonumber
			(\delta_x^3 \mathfrak{u})_0 = \mathfrak{u}_0^{(3)}, \quad (\delta_x^3 \mathfrak{u})_N = \mathfrak{u}_N^{(3)}.
		\end{gather}
		
		Similarly, we re-define the fourth-order discrete derivative operator $\delta_x^4$ in Equation \eqref{eq_delta4_op} by adding the boundary values $\mathfrak{u}_0^{(4)}, \mathfrak{u}_N^{(4)}$ to be determined below
		\begin{gather} \label{eq_delta4_op_new}
			(\delta_x^4 \mathfrak{u})_j =
			\frac{12}{h^2}((\delta_x\widetilde{\delta_x}\mathfrak{u})_j-(\delta_x^2\mathfrak{u})_j), \quad \forall j\in\{1,...,N-1\},\\ \nonumber
			(\delta_x^4 \mathfrak{u})_0 = \mathfrak{u}_0^{(4)}, \quad (\delta_x^4 \mathfrak{u})_N = \mathfrak{u}_N^{(4)}.
		\end{gather}

		Note that in the definition of $\delta_x^3 \mathfrak{u}$ at the near-boundary points $j=1,N-1$ the boundary values of $\mathfrak{u}_l^{(2)}=\widetilde{\delta_x^2}\fu$  for $l=0,N$  are needed. As remarked earlier, this need forces us (by the coupled equations below) also to associate discrete values to the third and fourth order derivatives on the boundary.

We now proceed to determine these values in terms of the discrete function $\mathfrak{u}$ and its Hermitian derivative at interior points. There will be three equations for the three unknowns at either one of the boundary points. For the point $x_0$ they are the following (analogous equations hold for $x_N$).
		
		\begin{equation} \label{eq_boundary_ode}
			f_0^\ast = \mathfrak{u}_0^{(4)} + D_0^\ast \mathfrak{u}_0^{(3)} + A_0^\ast \mathfrak{u}_0^{(2)},
		\end{equation}
		\begin{equation} \label{eq_boundary_simpson}
			\frac{\mathfrak{u}_2^{(3)}-\mathfrak{u}_0^{(3)}}{2h}=\frac{\mathfrak{u}_0^{(4)}+4\mathfrak{u}_1^{(4)}+\mathfrak{u}_2^{(4)}}{6},
		\end{equation}
		\begin{equation} \label{eq_boundary_lincomb}
			\mathfrak{u}_1^{(4)} = 2\frac{\mathfrak{u}_0^{(2)}-2\mathfrak{u}_1^{(2)}+\mathfrak{u}_2^{(2)}}{h^2} - \frac{\mathfrak{u}_2^{(3)}-\mathfrak{u}_0^{(3)}}{2h}.
		\end{equation}
		
		The unknowns in Equations \eqref{eq_boundary_ode} - \eqref{eq_boundary_lincomb} are $\mathfrak{u}_0^{(2)}, \mathfrak{u}_0^{(3)}$ and $\mathfrak{u}_0^{(4)}$. The remaining values $\mathfrak{u}^{(2)}_{1,2}, \mathfrak{u}_2^{(3)} $ and $\mathfrak{u}^{(4)}_{1,2}$ are all evaluated in terms of $\mathfrak{u}$ using equations \eqref{eqdefdelta2tild}, \eqref{eqdef3rd} and \eqref{eq_delta4_op}. Remark that the boundary value $\mathfrak{u}_0^{(2)}$ is incorporated in the evaluation of the near-boundary value $\mathfrak{u}_1^{(3)},$ as in ~\eqref{eqdef3rd}.

		The three equations above are derived from very basic facts. The first equation is simply Equation \eqref{utxxxxcontgen} for $x=0$, where the boundary conditions imply $\mathfrak{u}_0=\mathfrak{u}_0^{(1)}=0.$

		The second equation is a result of the Simpson integration scheme. That is to say, for some smooth $v:\Omega\rightarrow\mathbb{R}$ we know that
		\begin{equation*} \label{eq_simpson_integration}
			v(x_2) - v(x_0) = \int_{x_0}^{x_2} v'(x) dx = \frac{2h}{6}(v'(x_0) + 4v'(x_1) + v'(x_2)) + O(h^5).
		\end{equation*}

		Using $v = u^{(3)}$ yields
		\begin{equation} \label{eq_boundary_simpson_continuous}
			u^{(3)}(x_2) - u^{(3)}(x_0) = \frac{2h}{6}(u^{(4)}(x_0) + 4u^{(4)}(x_1) + u^{(4)}(x_2)) + O(h^5).
		\end{equation}

		A discrete approximation of this is Equation \eqref{eq_boundary_simpson}.

		The third equation is derived by a linear combination of two approximations of the second-order derivative as follows.\\

		For some smooth $v:\Omega\rightarrow\mathbb{R}$ we know that
		\begin{align*}
			v''(x_1) &= \frac{v(x_0)-2v(x_1)+v(x_2)}{h^2} -\frac{h^2}{12}v^{(4)}(x_1)+O(h^4),\\
			v''(x_1) &= \frac{v'(x_2)-v'(x_0)}{2h} -\frac{h^2}{6}v^{(4)}(x_1)+O(h^4).
		\end{align*}

		A linear combination of these equations gives
		\begin{equation*}
			v''(x_1) = 2\frac{v(x_0)-2v(x_1)+v(x_2)}{h^2} - \frac{v'(x_2)-v'(x_0)}{2h} + O(h^4).
		\end{equation*}

		Substituting $v = u^{(2)}$ we get
		\begin{equation} \label{eq_boundary_lincomb_continuous}
			u^{(4)}(x_1) = 2\frac{u^{(2)}(x_0)-2u^{(2)}(x_1)+u^{(2)}(x_2)}{h^2} - \frac{u^{(3)}(x_2)-u^{(3)}(x_0)}{2h} + O(h^4)
		\end{equation}

		Equation \eqref{eq_boundary_lincomb} is a discrete approximation of this equation.

		To summarize, the supplementary system of equations for the boundary values is as follows. \\

		\fbox{\parbox{\textwidth}{

			Equation \eqref{eq_boundary_ode} is the ODE itself.

			Equation \eqref{eq_boundary_simpson} represents the relation between 3rd and 4th-order derivatives by the Simpson operator, derived from the fundamental theorem of calculus.

			Equation \eqref{eq_boundary_lincomb} represents the connection between the 2nd, 3rd and 4th-order derivatives using Taylor's theorem.

		}}\bigskip
		
		Note that the equations only have a solution if the following conditions hold, which must be the case for sufficiently small values of $h$:
		\begin{align}
			\label{eq_conditions1}
			12 - 4 D_0 h + A_0 h^2 &\ne 0 \quad \text{(for $x=0$)} \\
			\label{eq_conditions2}
			12 + 4 D_N h + A_N h^2 &\ne 0 \quad \text{(for $x=1$)}
		\end{align}

		\subsection{Truncation error estimates}
		\label{sub_section_error_analysis}
		
			A full study of the truncation errors of the various differential operators appears in Appendix \ref{appendix_truncation}. This study shows different orders of the truncation error for interior, near-boundary and boundary points as summarized in Equation ~\eqref{eqtrunc1st}, as well as the following Equations ~\eqref{eqtrunc2nd}, ~\eqref{eqtrunc3rd} and ~\eqref{eqtrunc4th}.

			\begin{itemize}

				\item The truncation error for $\widetilde{\delta_x^2}$ (see \eqref{eq_deltax2_tilde_trunc}, \eqref{eq_deltax2_tilde_trunc_near_boundary} and \eqref{boundary_values_truncation_error}) is
				\begin{equation} \label{eqtrunc2nd}
				(\widetilde{\delta_x^2} u^\ast)_j = u^{(2)}(x_j) + \left\{\begin{matrix}
				O(h^4) &\forall j\in\{2,...,N-2\}, \\
				O(h^3) &\forall j\in\{0, 1, N-1, N\}.
				\end{matrix}\right.
				\end{equation}

				\item The truncation error for $\delta_x^3$ (see \eqref{eq_deltax3_trunc}, \eqref{boundary_values_truncation_error} and \eqref{eq_deltax3_trunc_full}) is
				\begin{equation} \label{eqtrunc3rd}
				(\delta_x^3 u^\ast)_j = u^{(3)}(x_j) + \left\{\begin{matrix}
				O(h^4) &\forall j\in\{2,...,N-2\}, \\
				O(h^2) &\forall j\in\{0, 1, N-1, N\}.
				\end{matrix}\right.
				\end{equation}

				\item The truncation error for $\delta_x^4$ (see \eqref{eq_deltax4_trunc}, \eqref{eq_deltax4_trunc_near_boundary} and \eqref{boundary_values_truncation_error}) is
				\begin{equation} \label{eqtrunc4th}
					(\delta_x^4 u^\ast)_j = u^{(4)}(x_j) + \left\{\begin{matrix}
						O(h^4) &\forall j\in\{2,...,N-2\}, \\
						O(h) &j = 1, N-1,\\
						O(h^2) &j = 0, N.
					\end{matrix}\right.
				\end{equation}
			\end{itemize}

	\section{\textbf{COMPACTNESS-THE DISCRETE VERSION OF RELLICH'S THEOREM}}\label{seccompact}
	
		The compactness of the inverse of an elliptic operator is equivalent (by domain considerations) to the compact embedding of the Sobolev space $H^s,\,s>0$ in $L^2.$ This is the celebrated Rellich theorem ~\cite[Chapter 5.7]{evans}, which is the cornerstone of the elliptic theory. Its proof requires several tools (for example, in a popular version of the proof, the use of Fourier transform and the Arzela-Ascoli theorem).
		
		In the discrete framework we do not have some of the aforementioned analytical tools. Yet we can ask ourselves the following question.
		
		\textbf{QUESTION: Is there a suitable $``\mbox{compactness}''$ property of the inverse $\Big(\delta^4_x\Big)^{-1}$?}
		
		Of course, if we just consider a \textit{fixed} $h>0$ such a question is meaningless since the underlying space is finite dimensional. However, we can provide a meaningful answer if \textit{all values} of $h>0$ are considered. In some sense, the compactness property is related to an ``increasing sequence of finite-dimensional spaces''.
		
		This question was addressed in ~\cite{IMA_SL_paper} and the answer is given below In Theorem ~\ref{thmcompactinverse} and Corollary ~\ref{cordelta2delta4}.
		
		We first introduce some notation, basically relating grid functions to functions defined on the interval $\Omega=[0,1]$ (see ~\cite[Section 10.2]{book}):
		
		For a grid function $\fz\in l^2_{h,0}$ we define its associated piecewise linear continuous function by
		
		\begin{defn}\label{defP1c}
			$$z_h(x)=\begin{cases}\mbox{linear in the interval}\,\, K_{i+\frac12}=(x_i,x_{i+1}),\,\,0\leq i\leq N-1, \\
			\fz_i,\,x=x_i,\,\,0\leq i\leq N.
			\end{cases}$$
		\end{defn}
		We now cite Theorem 3.7 in ~\cite{IMA_SL_paper}.
		\begin{thm} \label{thmcompactinverse}[\textbf{The discrete Rellich theorem}]
			Let $\set{0<N_1<N_2<...N_k<...}$ be an increasing sequence of integers and denote $h_k=\frac{1}{N_k},\,k=1,2,...$ Let $\set{\fv^{(k)}\in l^2_{h_k,0},\,k=1,2,...}$ be a bounded sequence of vectors so that
			\begin{equation}
			\label{eqth3.6bddv}\sup\set{|\fv^{(k)}|_{h_k},\,k=1,2,...}<\infty,
			\end{equation}
			and let
			\begin{equation*}
			\set{\fg^{(k)}=\Big(\delta^4_x\Big)^{-1}(\fv^{(k)}),\,\,k=1,2,...}.
			\end{equation*}
			
			Let $\set{g_{h_k},\,v_{h_k}}_{k=1}^\infty$ be the piecewise linear continuous functions in $\Omega=[0,1]$ corresponding to $\set{\fg^{(k)},\,\fv^{(k)}}_{k=1}^\infty,$ respectively (Definition ~\ref{defP1c}).
			
			In addition, let $\set{\fg^{(k)}_x}_{k=1}^\infty$ be the sequence of Hermitian derivatives of $\set{\fg^{(k)}}_{k=1}^\infty$ and let $\set{p_{h_k}}_{k=1}^\infty$ be the piecewise linear continuous functions in $\Omega=[0,1]$ corresponding to $\set{\fg^{(k)}_x}_{k=1}^\infty.$
			
			Then there exist subsequences $$\set{g_j:=g_{h_{k_j}},\,p_j:=p_{h_{k_j}},\,v_j:=v_{h_{k_j}}}_{j=1}^\infty$$ and limit functions $g(x),\,p(x),\,v(x),$ such that
			\begin{equation} \label{equnifconvgj}
				\lim\limits_{j\to\infty}g_j(x)=g(x)\,\,\mbox{in}\,\,\,C(\Omega),
			\end{equation}
			\begin{equation} \label{equnifconvpj}
				\lim\limits_{j\to\infty}p_j(x)=p(x)\,\,\mbox{in}\,\,\,C(\Omega),
			\end{equation}
			\begin{equation}
				\lim\limits_{j\to\infty}v_j(x)=v(x)\,\,\mbox{weakly in}\,\,\,L^2(\Omega).
			\end{equation}
			
			The limit function $g(x)$ is in $H^4(\Omega)\cap H^2_0(\Omega)$ and its derivatives satisfy
			\begin{equation} \label{eqthmg4isv}
				g'(x)=p(x),\quad\Big(\parx\Big)^4 g(x)=v(x).
			\end{equation}
		\end{thm}
		
		\bigskip
		
		In Theorem ~\ref{thmcompactinverse} we have seen that in addition to the convergence ~\eqref{equnifconvgj}, the piecewise-linear functions corresponding to the Hermitian derivatives $\set{\fg^{(k_j)}_x}_{j=1}^\infty$ converge uniformly to $g'(x)$ ~\eqref{equnifconvpj}. Next we repeat Corollary 3.8 in ~\cite{IMA_SL_paper} that yields a weaker convergence statement for the second-order derivatives. We shall need details of its proof when dealing with convergence properties of the \textit{third-order derivative} in Corollary ~\ref{cordelta3conv} below. Recall that the third-order derivative has not been addressed in the Sturm-Liouville (self-adjoint) case.
		
		\begin{cor} \label{cordelta2delta4}
			In the setting of Theorem ~\ref{thmcompactinverse} let
			$$ \fw^{(k)}=\widetilde{\delta_x^2}\fg^{(k)}
			=\widetilde{\delta}_x^2\Big(\delta^4_x\Big)^{-1}(\fv^{(k)}).$$
			
			Let $w_{h_k}$ be the piecewise linear continuous functions in $\Omega=[0,1]$ corresponding to $\fw^{(k)}$ (Definition ~\ref{defP1c}).
			
			Let the sequences $\set{g_j:=g_{h_{k_j}},\,v_j:=v_{h_{k_j}}}_{j=1}^\infty$ and limit functions $g(x),\,\,v(x),$ be as in theorem ~\ref{thmcompactinverse} and let $\set{w_j:=w_{h_{k_j}}}_{j=1}^\infty.$
			
			Then
			\begin{equation} \label{eqlimwj}
				\lim\limits_{j\to\infty}w_j(x)=g''(x)\,\,\mbox{weakly in}\,\,\,L^2(\Omega).
			\end{equation}
		\end{cor}
		
		\begin{proof}
			In light of Proposition ~\ref{propbounded} we have
			\begin{equation}
				\label{equnifbddfw}\sup\limits_{k=1,2,\ldots}\set{|\fw^{(k)}|_{h_{k}}} <\infty.
			\end{equation}
			
			Let $\phi(x)\in C^\infty_0(0,1)$ be a test function . Then
			$$
			(\fw^{(k_j)},\phi^\ast)_{h_{k_j}}=(\widetilde{\delta}_x^2\fg^{(k_j)},\phi^\ast)_{h_{k_j}}=
			(\fg^{(k_j)},\widetilde{\delta}_x^2\phi^\ast)_{h_{k_j}}.
			$$
			The sum in the left-hand side, can be expressed in terms of the corresponding piecewise linear functions as (see ~\cite[Lemma 10.4]{book})
			
			\begin{equation} \label{eqwjphi2}
				\int_0^1w_j(x)\phi_j(x)dx=			(\fw^{(k_j)},\phi^\ast)_{h_{k_j}}			-\frac{h_{k_j}}{6}\suml_{m=0}^{N_{k_j}-1}			(\fw^{(k_j)}_{m+1}-\fw^{(k_j)}_{m})(\phi^\ast_{m+1}-\phi^\ast_{m})
			\end{equation}
			
			where $\phi^\ast\in l^2_{h_{k_j},0}$ is the grid function associated with the function $\phi(x)$ (with mesh size $h_{k_j}$).
			
			By the Cauchy-Schwarz inequality
			$$|\suml_{m=0}^{N_{k_j}-1}
			(\fw^{(k_j)}_{m+1}-\fw^{(k_j)}_{m})(\phi^\ast_{m+1}-\phi^\ast_{m})|\leq 2N_{k_j}|\fw^{(k_j)}|_{h_{k_j}}\max\limits_{0\leq m\leq N_{k_j}-1} |\phi^\ast_{m+1}-\phi^\ast_{m}|.$$
			
			Clearly
			$$\max\limits_{0\leq m\leq N_{k_j}-1} |\phi^\ast_{m+1}-\phi^\ast_{m}|\xrightarrow[j\to\infty]{}0,$$
			
			hence
			\begin{equation} \label{eqintwjphij}
				\lim\limits_{j\to\infty}\int_0^1w_j(x)\phi_j(x)dx=		\lim\limits_{j\to\infty}(\fw^{(k_j)},\phi^\ast)_{h_{k_j}} =\lim\limits_{j\to\infty}(\fg^{(k_j)},\widetilde{\delta_x^2}\phi^\ast)_{h_{k_j}}.
			\end{equation}
			
			The uniform boundedness of $\set{w_j}_{j=1}^\infty$ and the uniform convergence of $\set{\phi_j}$ to $\phi$ imply
			$$\lim\limits_{j\to\infty}\int_0^1w_j(x)\phi_j(x)dx=\lim\limits_{j\to\infty}\int_0^1w_j(x)\phi(x)dx,$$
			
			so that
			\begin{equation}
				\lim\limits_{j\to\infty}\int_0^1w_j(x)\phi(x)dx=\lim\limits_{j\to\infty}(\fg^{(k_j)},\widetilde{\delta_x^2}\phi^\ast)_{h_{k_j}}.
			\end{equation}
			
			We let $(\phi'')_j$ denote the piecewise linear continuous function corresponding to $\phi''$ (for mesh size $h_{k_j}$).
			
			By ~\eqref{eqtrunc2nd}
			$$\lim\limits_{j\to\infty}(\fg^{(k_j)},\widetilde{\delta_x^2}\phi^\ast)_{h_{k_j}}=
			\lim\limits_{j\to\infty}(\fg^{(k_j)},(\phi'')^\ast)_{h_{k_j}},$$
			
			and as in the first equality in ~\eqref{eqintwjphij}
			\begin{equation}
				\label{eqfgphi2}\lim\limits_{j\to\infty}(\fg^{(k_j)},(\phi'')^\ast)_{h_{k_j}} = \lim\limits_{j\to\infty}\int_0^1g_j(x)(\phi'')_j(x)dx = \int_0^1g(x)\phi''(x)dx.
			\end{equation}
			
			Thus
			$$\lim\limits_{j\to\infty}\int_0^1w_j(x)\phi(x)dx=\int_0^1g(x)\phi''(x)dx,$$
			
			which proves ~\eqref{eqlimwj}.
			
		\end{proof}
		
		We show next that a similar, but weaker, claim holds for the third-order discrete derivative, defined in ~\eqref{eqdef3rd}.
		
		In the setting of Theorem ~\ref{thmcompactinverse} let
		$$\fn^{(k)}=\delta_x^3\fg^{(k)}
		=\delta_x^3\Big(\delta^4_x\Big)^{-1}(\fv^{(k)}).$$
		
		Observe that unlike the uniform boundedness of the discrete second-order derivatives ~\eqref{equnifbddfw}, we do not have a similar claim for the sequence $\set{|\fn^{(k)}|_{h_k}}.$ That's essentially due to the fact that for the discrete third-order derivative $\delta_x^3$ we do not have a boundedness result analogous to Lemma ~\ref{lembounddel2del4sqrt}. Thus we cannot derive directly the analog of the limit statement in ~\eqref{eqintwjphij} and a more careful use of duality is needed.
		
		\begin{cor}\label{cordelta3conv}
			In the setting of Theorem ~\ref{thmcompactinverse} we have
			\begin{equation} \label{eqlimdeltax3}
				\lim\limits_{j\to\infty}(\fn^{(k_j)},\phi^\ast)_{h_{k_j}}=\int\limits_0^1g^{(3)}(x)\phi(x)dx,
			\end{equation}
			
			where $\phi(x)$ is a test function as in the proof of Theorem ~\ref{thmcompactinverse}.
		\end{cor}
		
		\begin{proof}
			Let $\phi(x)$ be a test function as in the proof of Theorem ~\ref{thmcompactinverse}. Then, with the notation used in that proof,
			\begin{equation} \label{eqdelta3phi}
				(\fn^{(k_j)},\phi^\ast)_{h_{k_j}}=(\delta_x^3\fg^{(k_j)},\phi^\ast)_{h_{k_j}}=(2\delta_x^2\fg^{(k_j)}_x			-\delta_x\widetilde{\delta}_x^2\fg^{(k_j)}			,\phi^\ast)_{h_{k_j}}.
			\end{equation}
			
			Let us now study separately the two terms in the right-hand side of ~\eqref{eqdelta3phi}.
			
			We have
			$$(2\delta_x^2\fg^{(k_j)}_x,\phi^\ast)_{h_{k_j}}=(2\fg^{(k_j)}_x,\delta_x^2\phi^\ast)_{h_{k_j}}.$$
			
			We let $(\phi'')_j$ denote the piecewise linear function corresponding to $\phi''$ (for mesh size $h_{k_j}$) and recall that $\set{p_j}$ are the piecewise linear functions corresponding to $\set{\fg_x^{(k_j)}}$ ~\eqref{equnifconvpj}.
			
			In analogy with ~\eqref{eqfgphi2} we get
			\begin{equation} \label{eqfgxphidelta2}
				\lim\limits_{j\to\infty}(\fg_x^{(k_j)},\delta_x^2\phi^\ast)_{h_{k_j}} = 	\lim\limits_{j\to\infty}\int_0^1p_j(x)\phi''_j(x)dx=\int_0^1g'(x)\phi''(x)dx,
			\end{equation}
			
			hence
			\begin{equation} \label{eq2delx2gx}
				\lim\limits_{j\to\infty}(2\delta_x^2\fg^{(k_j)}_x,\phi^\ast)_{h_{k_j}} = \int_0^1g^{(3)}(x)\phi(x)dx.
			\end{equation}
			
			Next we have
			$$(\delta_x\widetilde{\delta}_x^2\fg^{(k_j)}
			,\phi^\ast)_{h_{k_j}}=-(\widetilde{\delta}_x^2\fg^{(k_j)},\delta_x\phi^\ast)_{h_{k_j}}.$$
			
			Invoking Corollary ~\ref{cordelta2delta4} (note that $\widetilde{\delta}_x^2\fg^{(k_j)}=\fw^{(k_j)}$) yields
			\begin{equation} \label{eqdelta2xg}
				\lim\limits_{j\to\infty}(\widetilde{\delta}_x^2\fg^{(k_j)},\delta_x\phi^\ast)_{h_{k_j}} = 	\int_0^1g''(x)\phi'(x)dx=-\int_0^1g^{(3)}(x)\phi(x)dx.
			\end{equation}
			
			Plugging ~\eqref{eq2delx2gx} and ~\eqref{eqdelta2xg} into ~\eqref{eqdelta3phi} we obtain ~\eqref{eqlimdeltax3}.
		\end{proof}

	\section{\textbf{A DISCRETE VERSION OF THE FOURTH-ORDER EQUATION}} \label{secdiscreteconvergence}

		Recall the operator $L_{cont}$ in \eqref{utxxxxcontgen}. Unlike the differential operator handled in \cite{IMA_SL_paper}, this one is not self-adjoint, which means that its eigenvalues are generally not real (in which case they appear as complex-conjugate pairs ~\cite{Naimark_book}).

		Using the finite difference operators introduced in Sections ~\ref{secdiscsetup} and ~\ref{secboundary}, and taking $h=\frac{1}{N},$ we introduce the discrete analog of Equation ~\eqref{utxxxxcontgen} by
		\begin{equation} \label{utxxxxdiscgen} \aligned
			(L_{disc,h}\fg^h)_i=(\delta^4_x\fg^h)_i+D^{\ast,h}_i(\delta_x^3\fg^h)_i + A^{\ast,h}_i(\widetilde{\delta_x^2}\fg^h)_i+(A')^{\ast,h}_i(\fg^h_x)_i\\+H^{\ast,h}_i(\fg^h_x)_i + B^{\ast,h}_i\fg^h_i=f^{\ast,h}_i,\quad 1\leq i\leq N-1,
		\endaligned \end{equation}

		where $f^{\ast,h},\,\,D^{\ast,h},\,A^{\ast,h},\,\,(A')^{\ast,h},\,\,H^{\ast,h},\,\,\,B^{\ast,h}$ are the grid functions corresponding, respectively, to $f(x),\,\,D(x),\,\,A(x),\,A'(x),\,\,H(x),\,\, B(x).$

		We assume that $f(x)$ is continuous in $\Omega=[0,1].$

		The equation is supplemented with homogeneous boundary conditions
		$$\fg^h_0=(\fg^h_x)_0=\fg^h_N=(\fg^h_x)_N=0.$$

		Thus, we seek solution $\fg^h\in l^2_{h,0},$ such that also $\fg^h_x\in l^2_{h,0}.$

		\begin{rem}
			As in Remark ~\ref{remzerodiscderiv} we assume that all grid functions and their Hermitian derivatives are in $l^2_{h,0}.$ This amounts simply to extending the grid functions (whose relevant values are at the interior points $\set{x_i,\,\,1\leq i\leq N-1}$) as zero at the endpoints $x_0,x_N.$
		\end{rem}

		In what follows we designate,
		\begin{equation} \label{eqnoteziwi}
			\begin{cases}

				\fg^h_x,\quad \mbox{the Hermitian derivative of} \,\,\fg^h\,, \\

				\fv^h=\delta^4_x\fg^h, \\

				\fw^h=\widetilde{\delta}_x^2\fg^h = \widetilde{\delta}_x^2\Big(\delta^4_x\Big)^{-1}\fv^h \\

				\fn^h=\delta_x^3\fg^h = \delta_x^3\Big(\delta^4_x\Big)^{-1}\fv^h.

			\end{cases}
		\end{equation}

		The basic result here is that ``stability'' implies ``convergence'' as follows.

		\begin{thm} \label{thmconvgeneralscheme}[\textbf{General convergence}]
			Let $\set{0<N_1<N_2<...N_k<...}$ be an increasing sequence of integers and denote $h_k=\frac{1}{N_k},\,k=1,2,...$

			Let $\set{\fg^{(k)}=\fg^{h_k}\in l^2_{h_k,0},\,k=1,2,...}$ be a sequence of solutions to Equation ~\eqref{utxxxxdiscgen} (with $h=h_k).$ Let $\fv^{(k)}=\fv^{h_k}$ and assume that $\fv_x^{(k)}\in l^2_{h_k,0},\,k=1,2,...$

			Assume that
			\begin{equation} \label{eqbounddeltax4g}
				\sup\set{|\fv^{(k)}|_{h_k}=|\delta^4_x\fg^{(k)}|_{h_k},\,k=1,2,...}<\infty.
			\end{equation}

			Let $g_{h_k},\,v_{h_k}$ be the piecewise linear continuous functions in $\Omega=[0,1]$ corresponding to $\fg^{(k)},\,\fv^{(k)}$ (Definition ~\ref{defP1c}).

			Then these sequences converge to limit functions $g(x),\,\,v(x),$ in the following sense
			\begin{equation}
				\lim\limits_{k\to\infty}g_{h_k}(x)=g(x)\,\,\mbox{in}\,\,\,C(\Omega),
			\end{equation}
			\begin{equation}
				\lim\limits_{k\to\infty}v_{h_k}(x)=v(x)\,\,\mbox{weakly in}\,\,\,L^2(\Omega).
			\end{equation}

			The limit function $g(x)$ is in $H^4(\Omega)\cap H^2_0(\Omega)$ and satisfies Equation ~\eqref{utxxxxcontgen}:

			$$\aligned L_{cont}g=\Big(\parx\Big)^4g+D(x)\Big(\parx\Big)^3g+
			A(x)\Big(\parx\Big)^2g+(A'(x)+H(x))\Big(\parx\Big)g\\+B(x)g=f(x),\quad
			x\in\Omega=[0,1],
			\endaligned$$			
		\end{thm}

		\begin{proof}

			Using the notation
			$$\fw^{(k)}=\fw^{h_k},\,\fn^{(k)}=\fn^{h_k}$$

			and taking the scalar product of Equation ~\eqref{utxxxxdiscgen} with $\phi^\ast$ in $l^2_{h_k,0}$ yields
			\begin{equation} \label{eqdischk} \aligned
				(\delta^4_x\fg^{(k)},\phi^\ast)_{h_k} + (D^{\ast,h_k}\fn^{(k)},\phi^\ast)_{h_k} + (A^{\ast,h_k}\fw^{(k)},\phi^\ast)_{h_k} \\
				+ (((A')^{\ast,h_k}+H^{\ast,h_k}))\fg^{(k)}_x,\phi^\ast)_{h_k} + (B^{\ast,h_k}\fg^{(k)},\phi^\ast)_{h_k}=(f^{\ast,h_k},\phi^\ast)_{h_k}.
			\endaligned \end{equation}

			Note that in the above equation the simplified vector notation			

			$D^{\ast,h_k}\fn^{(k)}=(D^{\ast,h_k}_1\fn^{(k)}_1,\ldots,D^{\ast,h_k}_{N-1}\fn^{(k)}_{N-1})$

			has been used (similarly for the other terms).

			The boundedness assumption ~\eqref{eqbounddeltax4g} enables us to invoke Theorem ~\ref{thmcompactinverse} and Corollary ~\ref{cordelta2delta4}. Thus there exist subsequences $\set{g_j:=g_{h_{k_j}},\,v_j:=v_{h_{k_j}}}_{j=1}^\infty$ and limit functions $g(x),\,\,v(x),$ such that
			\begin{equation}
				\begin{cases}
					\lim\limits_{j\to\infty}g_j(x)=g(x)\,\,\mbox{in}\,\,\,C(\Omega), \\
					\lim\limits_{j\to\infty}v_j(x)=v(x)\,\,\mbox{weakly in}\,\,\,L^2(\Omega).
				\end{cases}
			\end{equation}

			The limit function $g(x)$ is in $H^4(\Omega)\cap H^2_0(\Omega)$ and $\Big(\parx\Big)^4g=v.$

			Denote by $p_{h_k}\,,\,w_{h_k}$ the piecewise linear continuous functions in $\Omega=[0,1]$ corresponding, respectively, to $\fg^{(k)}_x\,,\fw^{(k)}$ (Definition ~\ref{defP1c}). Let $$\set{p_j=p_{h_{k_j}},\,w_j:=w_{h_{k_j}}}_{j=1}^\infty.$$

			From ~\eqref{equnifconvpj} and ~\eqref{eqlimwj} we obtain,
			\begin{equation} \label{eqlimwjconv}
				\begin{cases}
					\lim\limits_{j\to\infty}p_j(x)=g'(x)\,\,\mbox{ in}\,\,\,C(\Omega), \\
					\lim\limits_{j\to\infty}w_j(x)=g''(x)\,\,\mbox{weakly in}\,\,\,L^2(\Omega).
				\end{cases}
			\end{equation}

			Now consider the term $$(D^{\ast,h_{k_j}}\fn^{(k)},\phi^\ast)_{h_{k_j}}=
			(\fn^{(k_j)},(D\phi)^\ast)_{h_{k_j}}.$$

			From Corollary ~\ref{cordelta3conv} we infer
			\begin{equation} \label{eqlimfnDphi}
				\lim\limits_{j\to\infty}(\fn^{(k_j)},(D\phi)^\ast)_{h_{k_j}} = \int\limits_0^1D(x)g^{(3)}(x)\phi(x)dx.
			\end{equation}

			Inserting these limits in ~\eqref{eqdischk} we conclude that the equation $L_{cont}g=f$ is satisfied in the weak sense.

			However, in view of the Assumption ~\ref{assumespLcont} there is a unique solution to this equation, so all subsequences of $\set{g_{h_k},\,v_{h_k}}_{k=1}^\infty$ converge to the same limit. This concludes the proof of the theorem.

		\end{proof}

		A rigorous proof of the optimal convergence rate in the general case is currently missing. However, it has been established for the constant coefficient self-adjoint case (i.e. $D(x)=H(x)\equiv 0$ and constant coefficients $A,\,B$) in \cite[Theorem 5.7]{IMA_SL_paper}. The convergence rates of the first and second-order derivatives $\mathfrak{u}_x = \widetilde{\delta_x} \mathfrak{u}$ and $\widetilde{\delta_x^2} \mathfrak{u}$ to the exact solutions ${u^{(1)}}^\ast$ and ${u^{(2)}}^\ast$ on $\{x_1, ..., x_{N-1}\}$ were proved to be $\frac{27}{8}$, and $\frac{11}{4}$ respectively ~\cite[Corollary 5.9]{IMA_SL_paper}.
	
	\section{\textbf{NUMERICAL  EXAMPLES}} \label{secnumerresults}

		In this section we present two numerical examples that serve to illustrate the effectiveness of the discrete approximation as expressed in Equation ~\eqref{utxxxxdiscgen}. The first example is a rather simple, constant coefficient equation that serves to demonstrate the high order accuracy of the scheme presented here, even for non-homogeneous boundary data. The second example illustrates the accuracy of the method when the coefficients are non-constant, even for highly oscillating solutions.

		For the convenience of the reader, we summarize here the \textbf{notation} introduced in the preceding sections. It will serve below in order to give a clear definition of ``truncation errors'' and ``accuracy errors'', and will also be used  in the presentation of the examples.
		\begin{itemize}
			\item $\{ {u^\ast, u^{(1)}}^\ast, {u^{(2)}}^\ast, {u^{(3)}}^\ast, {u^{(4)}}^\ast \}  $  are the grid functions corresponding to the exact solution and its exact consecutive derivatives.
			\item $\{\widetilde{\delta_x} u^\ast, \widetilde{\delta_x^2} u^\ast, \delta_x^3 u^\ast, \delta_x^4 u^\ast\}$ are the grid functions corresponding to the discrete derivatives of $u^\ast,$ as obtained by the application of the discrete difference operators.
		\end{itemize}

		Observe that the above discrete functions, all related to the exact solution $u^\ast,$ are never computed and are only invoked in error estimates, as discussed below.

		Our algorithm is a finite-difference scheme, aimed at computing the discrete solution $\fu= L_{disc, h}^{-1} f^{\ast}$ approximating the exact (and unknown) $u^\ast$ (see Equation \eqref{utxxxxdiscgen} where it is designated as $\fg^h$). It also provides  the consecutive discrete derivatives of $\fu,$ namely  $\widetilde{\delta_x} \mathfrak{u}$, $\widetilde{\delta_x^2} \mathfrak{u}$, $\delta_x^3 \mathfrak{u}$ and $\delta_x^4 \mathfrak{u}$.
		
		Let us clarify the following interpretation of discrete errors, for a given mesh size $h$:
		\begin{itemize}
		\item
			\textit{Truncation errors} refer to the differences between $\{ {u^{(1)}}^\ast, {u^{(2)}}^\ast, {u^{(3)}}^\ast, {u^{(4)}}^\ast \}$ and $\{\widetilde{\delta_x} u^\ast, \widetilde{\delta_x^2} u^\ast, \delta_x^3 u^\ast, \delta_x^4 u^\ast\}$ respectively.
			Note that such errors are calculus facts that do not rely on any differential equation. A detailed analysis of the truncation errors is supplied in Appendix \ref{appendix_truncation}.
		\item
			\textit{Accuracy errors} refer to the differences between $\{u^\ast, {u^{(1)}}^\ast, {u^{(2)}}^\ast, {u^{(3)}}^\ast, {u^{(4)}}^\ast \}$ and $\{\mathfrak{u}, \widetilde{\delta_x} \mathfrak{u}, \widetilde{\delta_x^2} \mathfrak{u}, \delta_x^3 \mathfrak{u}, \delta_x^4 \mathfrak{u}\}$ respectively, where as above $\mathfrak{u}$ is the solution to Equation \eqref{utxxxxdiscgen} (designated there as $\fg^h$ ).
		\end{itemize}
	
		The numerical examples show that our algorithm provides  highly accurate approximations to the exact solution and its derivatives $\{ {u^\ast, u^{(1)}}^\ast, {u^{(2)}}^\ast, {u^{(3)}}^\ast, {u^{(4)}}^\ast \}.$ Somewhat surprisingly, this approximation is more accurate than the respective \textit{truncation errors.}
		In fact we shall see that absolute errors are generally smaller by an order of magnitude (for a given mesh size $h$), and \textbf{the convergence rate at all points $\{x_0, ..., x_N\}$ is roughly 4} . We refer to the fourth-order convergence as ``optimal rate'' of convergence.

		The convergence rate of some norm error $\varepsilon$ is calculated as follows: Let $\varepsilon_1, \varepsilon_2$ be consecutive errors calculated with $N=N_1,N_2$ respectively. The convergence rate of $\varepsilon$ between these two calculations is then
		\begin{equation} \label{eq_convergence}
			\text{convergence-rate}_{1\to 2} = \log_{N_2/N_1}(\varepsilon_1/\varepsilon_2).
		\end{equation}

		\subsection{First test problem} \label{ex1_section}
		
		Consider Equation ~\eqref{utxxxxcontgen}, with constant coefficients $D=10, A=10^2, H=10^3, B=10^4$.
		
		The right-hand side is given by
		\begin{align*}
			f(x) = \{&\cos(2\pi x) [1-24\pi^2+16\pi^4+D(1-12\pi^2)+A(1-4\pi^2)+ H + B] \\ - 2\pi &\sin(2\pi x) [4-16\pi^2+D(3-4\pi^2)+ 2A + H]\} \cdot e^x.
		\end{align*}
		We solve the problem on the interval $\Omega=(0.3, 1.4)$, with the boundary conditions:
		\begin{gather*}
			u(0.3) = -0.417129, \quad u(1.4) = -3.28073,\\
			u'(0.3) = -8.48343, \quad u'(1.4) = -18.2572.
		\end{gather*}
		
		It is straightforward to check that the solution $u(x)$ is given by
		\begin{equation} \label{eq_ex1}
			u(x) = e^x \cos(2 \pi x),
		\end{equation}
		
		and its derivatives are
		\begin{align*}
			u^{(1)}(x) &= e^x (\cos(2\pi x) - 2 \pi \sin(2 \pi x)), \\
			u^{(2)}(x) &= e^x((1-4 \pi^2) \cos(2\pi x) - 4\pi \sin(2\pi x)), \\
			u^{(3)}(x) &= e^x((1-12 \pi^2) \cos(2\pi x) + 2\pi (4 \pi^2-3) \sin(2\pi x)), \\
			u^{(4)}(x) &= e^x((1-24 \pi^2+16 \pi^4) \cos(2\pi x) + 8\pi (4 \pi^2-1) \sin(2\pi x)).
		\end{align*}
		
		The grid values $N =8, 16, 32, 64$ and $128$ are used in subsequent solutions of the problem.
		
		The results with $N=32$ are plotted in Figure \ref{example1_plots}.
		
		The norms (both $|\cdot|_\infty$ and $|\cdot|_h$) of the truncation errors of the discrete difference operators, and their respective convergence rates are presented in Table \ref{tab_example1_trunc}.
		
		The norms of accuracy errors and their convergence rates are presented in Table \ref{tab_example1_err}.
		
		Finally, the pointwise convergence rates (at grid points) of the truncation errors, when comparing calculations with $N=32$ and $N=64$, are plotted in Figure \ref{ex1_cv_per_x}. Note that results at the boundary points $x_0, x_N$ match the expected convergence rates as per Section \ref{trunc_boundary}.

		\begin{figure}
			
			
			\includegraphics[width=\linewidth]{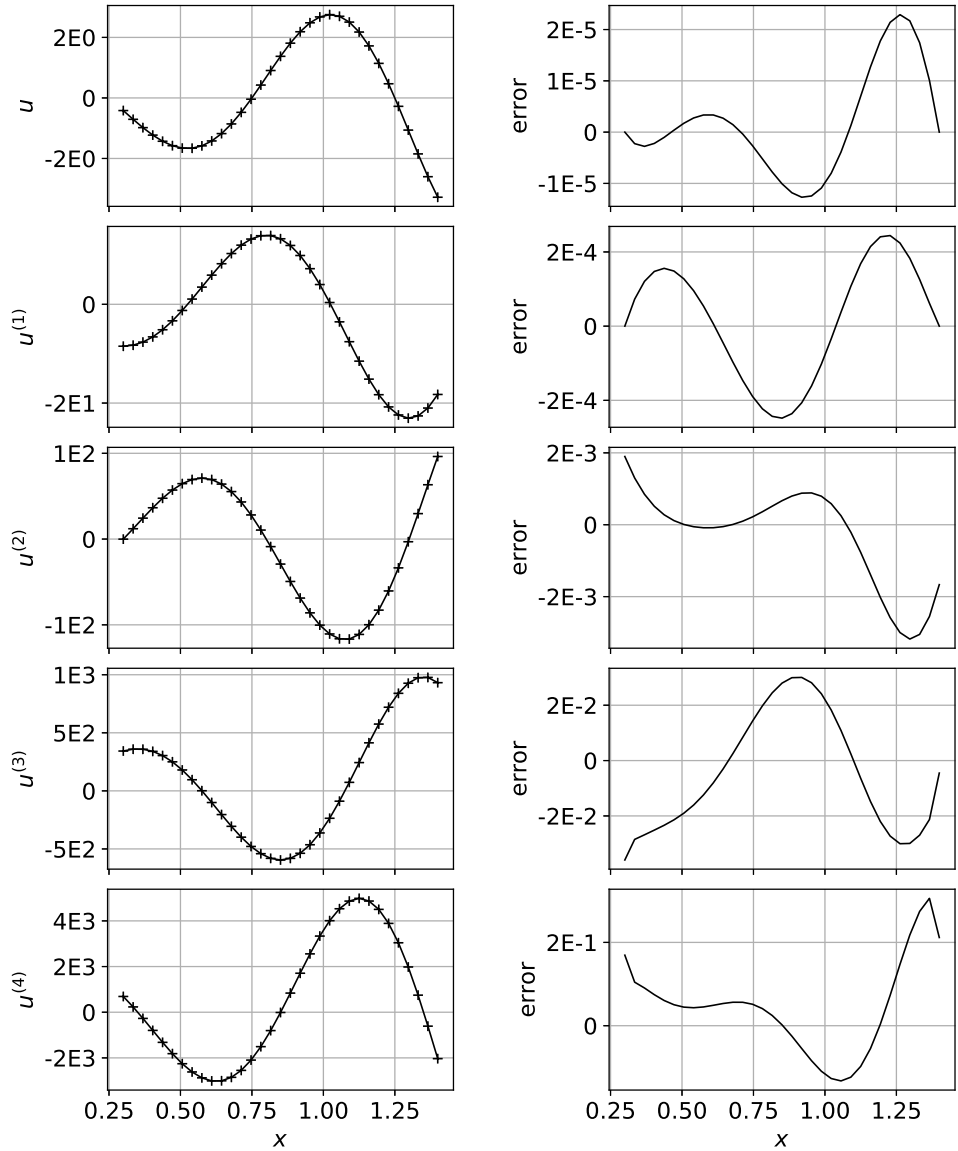}
			
			\caption{(left column) exact solution (solid line) and calculated discrete values (+ markers) of $u$ and its derivatives for the test problem in Subsection \ref{ex1_section}, and (right column) the difference between the exact solution and the calculated values, with $N=32$.}
			
			\label{example1_plots}
			
		\end{figure}
		
		{\setlength{\extrarowheight}{4pt}
			\begin{table}
				
				\begin{tabularx}{\linewidth}{
						|>{\hsize=1.4\hsize}X|
						>{\hsize=0.55\hsize}X|
						>{\hsize=0.55\hsize}X|
						>{\hsize=0.55\hsize}X|
						>{\hsize=0.55\hsize}X|
						>{\hsize=0.55\hsize}X|
					}
					\hline
					$N$ & 8 & 16 & 32 & 64 & 128 \\			
					\hline
					$|\widetilde{\delta_x}u^\ast-{u^{(1)}}^\ast|_h$
					& 4.40E-2 & 2.63E-3 & 1.62E-4 & 1.01E-5 & 6.32E-7 \\
					convergence rate & & 4.06 & 4.02 & 4.01 & 4.0 \\
					$|\widetilde{\delta_x}u^\ast-{u^{(1)}}^\ast|_\infty$
					& 8.33E-2 & 6.22E-3 & 4.05E-4 & 2.56E-5 & 1.60E-6 \\
					convergence rate & & 3.74 & 3.94 & 3.99 & 4.0 \\
					\hline
					$|\widetilde{\delta_x^2}u^\ast-{u^{(2)}}^\ast|_h$
					& 1.35E-1 & 1.19E-2 & 9.90E-4 & 8.32E-5 & 7.12E-6 \\
					convergence rate & & 3.5 & 3.59 & 3.57 & 3.55 \\
					$|\widetilde{\delta_x^2}u^\ast-{u^{(2)}}^\ast|_\infty$
					& 2.73E-1 & 2.93E-2 & 4.18E-3 & 5.41E-4 & 6.82E-5 \\
					convergence rate & & 3.22 & 2.81 & 2.95 & 2.99 \\
					\hline
					$|\delta_x^3 u^\ast-{u^{(3)}}^\ast|_h$
					& 6.52E+0 & 7.14E-1 & 1.00E-1 & 1.64E-2 & 2.83E-3 \\
					convergence rate & & 3.19 & 2.83 & 2.61 & 2.53 \\
					$|\delta_x^3 u^\ast-{u^{(3)}}^\ast|_\infty$
					& 1.37E+1 & 2.45E+0 & 4.89E-1 & 1.12E-1 & 2.70E-2 \\
					convergence rate & & 2.48 & 2.33 & 2.13 & 2.05 \\
					\hline
					$|\delta_x^4 u^\ast-{u^{(4)}}^\ast|_h$
					& 7.94E+1 & 2.59E+1 & 8.98E+0 & 3.16E+0 & 1.12E+0 \\
					convergence rate & & 1.62 & 1.53 & 1.51 & 1.50 \\
					$|\delta_x^4 u^\ast-{u^{(4)}}^\ast|_\infty$
					& 2.00E+2 & 9.11E+1 & 4.47E+1 & 2.22E+1 & 1.11E+1 \\
					convergence rate & & 1.13 & 1.03 & 1.01 & 1.0 \\
					\hline
				\end{tabularx}
				\caption{Truncation errors in the various discrete difference operators and their convergence rates, calculated for the test problem in Subsection \ref{ex1_section}. The convergence rates are calculated as per Equation \eqref{eq_convergence}.}
				
				\label{tab_example1_trunc}
				
			\end{table}
		}
		
		{\setlength{\extrarowheight}{4pt}%
			\begin{table}
				\begin{tabularx}{\linewidth}{
						|>{\hsize=1.4\hsize}X|
						>{\hsize=0.55\hsize}X|
						>{\hsize=0.55\hsize}X|
						>{\hsize=0.55\hsize}X|
						>{\hsize=0.55\hsize}X|
						>{\hsize=0.55\hsize}X|
					}
					
					\hline
					$N$ & 8 & 16 & 32 & 64 & 128 \\	
					\hline
					$|\mathfrak{u}-u^\ast|_h$
					& 2.66E-3 & 1.71E-4 & 1.07E-5 & 6.65E-7 & 4.14E-8 \\
					convergence rate & & 3.96 & 4.00 & 4.00 & 4.00 \\
					$|\mathfrak{u}-u^\ast|_\infty$
					& 6.07E-3 & 3.71E-4 & 2.29E-5 & 1.43E-6 & 8.90E-8 \\
					convergence rate & & 4.03 & 4.02 & 4.01 & 4.0 \\
					\hline
					$|\widetilde{\delta_x}\mathfrak{u}-{u^{(1)}}^\ast|_h$
					& 4.30E-2 & 2.66E-3 & 1.66E-4 & 1.03E-5 & 6.45E-7 \\
					convergence rate & & 4.01 & 4.01 & 4.00 & 4.00 \\
					$|\widetilde{\delta_x}\mathfrak{u}-{u^{(1)}}^\ast|_\infty$
					& 6.44E-2 & 3.99E-3 & 2.48E-4 & 1.55E-5 & 9.65E-7 \\
					convergence rate & & 4.01 & 4.01 & 4.00 & 4.00 \\
					\hline
					$|\widetilde{\delta_x^2}\mathfrak{u}-{u^{(2)}}^\ast|_h$
					& 4.33E-1 & 2.41E-2 & 1.45E-3 & 8.97E-5 & 5.56E-6 \\
					convergence rate & & 4.17 & 4.05 & 4.02 & 4.01 \\
					$|\widetilde{\delta_x^2}\mathfrak{u}-{u^{(2)}}^\ast|_\infty$
					& 7.56E-1 & 4.94E-2 & 3.18E-3 & 1.98E-4 & 1.24E-5 \\
					convergence rate & & 3.94 & 3.96 & 4.0 & 4.0 \\
					\hline
					$|\delta_x^3\mathfrak{u}-{u^{(3)}}^\ast|_h$
					& 6.44E+0 & 3.83E-1 & 2.31E-2 & 1.42E-3 & 8.84E-5 \\
					convergence rate & & 4.07 & 4.05 & 4.02 & 4.01 \\
					$|\delta_x^3\mathfrak{u}-{u^{(3)}}^\ast|_\infty$
					& 7.77E+0 & 6.27E-1 & 3.59E-2 & 2.08E-3 & 1.25E-4 \\
					convergence rate & & 3.63 & 4.12 & 4.11 & 4.06 \\
					\hline
					$|\delta_x^4\mathfrak{u}-{u^{(4)}}^\ast|_h$
					& 2.27E+1 & 1.92E+0 & 1.25E-1 & 7.80E-3 & 4.90E-4 \\
					convergence rate & & 3.56 & 3.95 & 4.00 & 3.99 \\
					$|\delta_x^4\mathfrak{u}-{u^{(4)}}^\ast|_\infty$
					& 3.08E+1 & 4.33E+0 & 3.06E-1 & 1.94E-2 & 1.21E-3 \\
					convergence rate & & 2.83 & 3.82 & 3.98 & 4.0 \\
					\hline
				\end{tabularx}
				
				\caption{Accuracy errors in the approximations to $\{u^\ast, {u^{(1)}}^\ast, {u^{(2)}}^\ast, {u^{(3)}}^\ast, {u^{(4)}}^\ast \}$ by $\mathfrak{u}$ (Equation \eqref{utxxxxdiscgen}) and its discrete derivatives $\{ \widetilde{\delta_x} \mathfrak{u}, \widetilde{\delta_x^2} \mathfrak{u}, \delta_x^3 \mathfrak{u}, \delta_x^4 \mathfrak{u} \}$, for the test problem in Subsection \ref{ex1_section}. The convergence rates are calculated as per Equation \eqref{eq_convergence}. These results are analogous to those presented in Table \ref{tab_example1_trunc}, but the discrete operators are now applied to the calculated function $\mathfrak{u}$ instead of the exact solution $u^\ast$. Note that the absolute errors and convergence rates are significantly improved as a result.}
				
				\label{tab_example1_err}
				
			\end{table}
		}
		
		\begin{figure}
			
			\includegraphics[width=\linewidth]{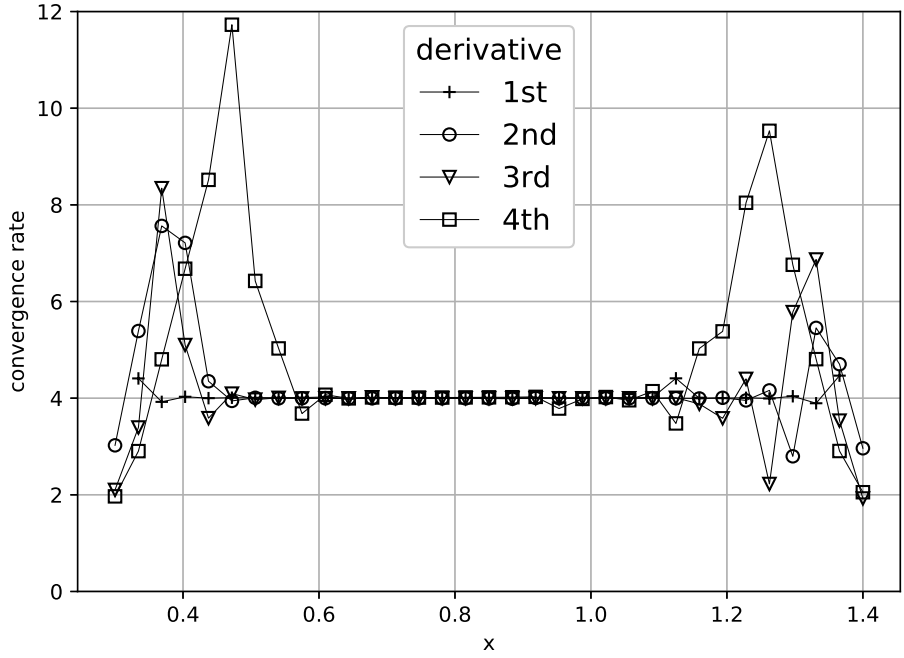}
			
			\caption{The pointwise truncation errors (at grid points) of $\widetilde{\delta_x} u^\ast, \widetilde{\delta_x^2} u^\ast, \delta_x^3 u^\ast$ and $\delta_x^4 u^\ast$ (compared to ${u^{(1)}}^\ast, {u^{(2)}}^\ast, {u^{(3)}}^\ast$ and ${u^{(4)}}^\ast$ respectively), for the test problem in Subsection \ref{ex1_section}. The convergence rates are obtained by comparing calculations with $N=32$ and $N=64$. The convergence rate at each point is calculated using Equation \eqref{eq_convergence}.}
			
			\label{ex1_cv_per_x}
			
		\end{figure}
		
		\clearpage
		\subsection{Second test problem} \label{ex2_section}
		
		We examine the following function, used in \cite{IMA_SL_paper} to test the numerical method on highly oscillatory functions:
		\begin{gather} \label{eq_ex2}
			u(x) = p(x) \sin(\frac{1}{q(x)+\varepsilon}), \\
			p(x) = x^2(1-x)^2, \quad q(x) = (x-\frac{1}{2})^2. \nonumber
		\end{gather}
		
		This boundary value problem is solved on the interval $\Omega=[0,1]$, so the boundary conditions are homogeneous (i.e. $u(0)=u(1)=u'(0)=u'(1)=0$).
		
		The parameter $\varepsilon$ controls oscillation frequency. In this example the value $\varepsilon = \frac{1}{40}$ is used, which provides a wavelength roughly comparable to $\varepsilon$.
		
		The coefficient functions used in \cite{IMA_SL_paper} are equivalent to
		\begin{equation} \label{eq_old_coefficients}
			\begin{aligned}
				\phantom{'} A(x) = \; & \alpha(1+\frac{1}{2}\sin(40\pi x)), \\
				A'(x) = & \; 20\; \pi\; \alpha\; \cos(40\pi x), \\
				\phantom{'}B(x) = \; & \beta\; \sin(40\pi x),
			\end{aligned}
		\end{equation}
		with $\alpha = 10^4$ and $\beta = 10^8$ chosen to ensure that the magnitudes of various terms in the differential equation are roughly equal. As the paper \cite{IMA_SL_paper} deals with the self-adjoint case, the remaining coefficients are
		\begin{equation*}
			H(x) = D(x) = 0. \\
		\end{equation*}
		
		In this work we can add the third-order derivative term, by taking these coefficient functions to be
		\begin{equation} \label{eq_new_coefficients}
			\begin{aligned}
			H(x) = & \; 0, \\
			D(x) = & \; \gamma\; \cos(40 \pi x),
			\end{aligned}
		\end{equation}
		with $\gamma=10^2$, again to ensure that the third-order term is comparable to the others. The coefficient functions in \eqref{eq_old_coefficients} remain unchanged.
		
		The detailed expression of $f(x)$ is too tedious to be fully presented here. The function itself is plotted in Figure \ref{ex2_f}.
		
		\begin{figure}
			
			\includegraphics[width=\linewidth]{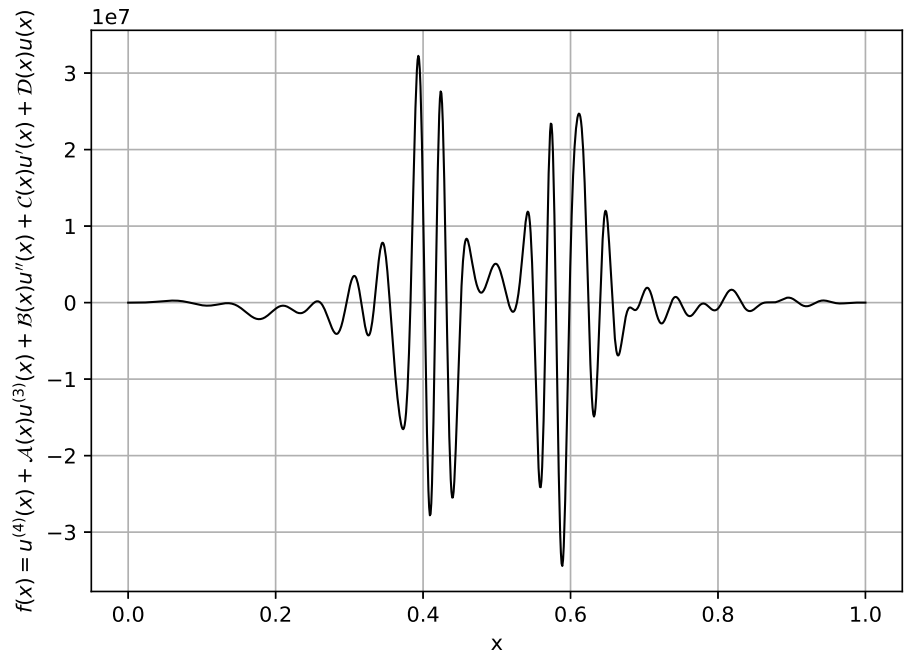}

			\caption{The function $f(x)$ obtained from Equation ~\eqref{utxxxxcontgen}, using $u$ as in Equation \eqref{eq_ex2}. The coefficients $A, B, D$ and $H$ are given in equations \eqref{eq_old_coefficients} and \eqref{eq_new_coefficients}.}
			
			\label{ex2_f}
			
		\end{figure}
		
		The calculated results with $N=512$ are plotted in Figure \ref{example2_plots}. The accuracy errors (with $|\cdot|_\infty$ and $|\cdot|_h$ norms) and their convergence rates are presented in Table \ref{tab_example2_err}.
		
		\begin{figure}
			
			\includegraphics[width=\linewidth]{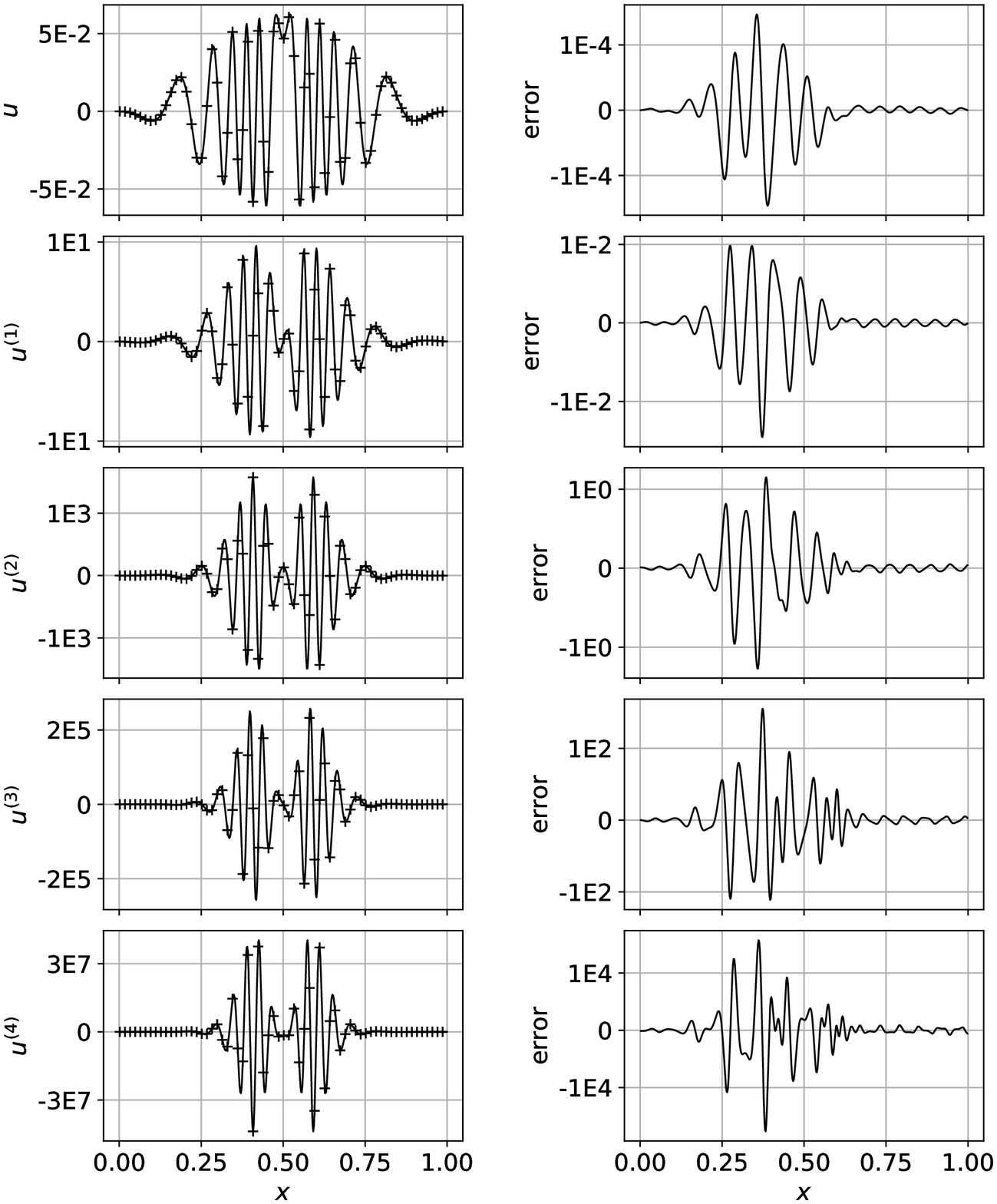}
			
			\caption{(left column) Exact solution (solid line) and calculated values (+ markers, where every eighth value is marked to avoid clatter) of $\fu$ and its discrete derivatives for the test problem in Subsection \ref{ex2_section}, with $N=512$. (right column) Difference between the exact solution and the calculated values.}
			
			\label{example2_plots}
			
		\end{figure}
		
		{\setlength{\extrarowheight}{4pt}
			\begin{table}
				\begin{tabularx}{\linewidth}{
						|>{\hsize=1.4\hsize}X|
						>{\hsize=0.55\hsize}X|
						>{\hsize=0.55\hsize}X|
						>{\hsize=0.55\hsize}X|
						>{\hsize=0.55\hsize}X|
						>{\hsize=0.55\hsize}X|
					}
					
					\hline
					$N$ & 32 & 128 & 512 & 2048 & 8192 \\	
					\hline
					$|\mathfrak{u}-u^\ast|_h$
					& 5.59E-1 & 9.27E-3 & 4.15E-5 & 1.64E-7 & 2.73E-9 \\
					convergence rate & & 2.96 & 3.90 & 3.99 & 2.95 \\
					$|\mathfrak{u}-u^\ast|_\infty$
					& 2.27 & 3.24E-2 & 1.47E-4 & 5.78E-7 & 7.63E-9 \\
					convergence rate & & 3.06 & 3.89 & 3.99 & 3.12 \\
					\hline
					$|\widetilde{\delta_x}\mathfrak{u}-{u^{(1)}}^\ast|_h$
					& 2.64E+1 & 7.86E-1 & 3.52E-3 & 1.38E-5 & 2.49E-7 \\
					convergence rate & & 2.54 & 3.90 & 3.99 & 2.90 \\
					$|\widetilde{\delta_x}\mathfrak{u}-{u^{(1)}}^\ast|_\infty$
					&9.71E+1 & 3.18 & 1.46E-2 & 5.74E-5 & 6.49E-7 \\
					convergence rate & & 2.47 & 3.88 & 3.99 & 3.23 \\
					\hline
					$|\widetilde{\delta_x^2}\mathfrak{u}-{u^{(2)}}^\ast|_h$
					& 2.64E+3 & 7.23E+1 & 3.21E-1 & 1.26E-3 & 2.48E-5 \\
					convergence rate & & 2.60 & 3.91 & 3.99 & 2.83 \\
					$|\widetilde{\delta_x^2}\mathfrak{u}-{u^{(2)}}^\ast|_\infty$
					& 9.14E+3 & 2.90E+2 & 1.27 & 5.03E-3 & 6.92E-5 \\
					convergence rate & & 2.49 & 3.92 & 3.99 & 3.09 \\
					\hline
					$|\delta_x^3\mathfrak{u}-{u^{(3)}}^\ast|_h$
					& 1.30E+5 & 7.83E+3 & 3.49E+1 & 1.37E-1 & 2.59E-3 \\
					convergence rate & & 2.03 & 3.90 & 3.99 & 2.86 \\
					$|\delta_x^3\mathfrak{u}-{u^{(3)}}^\ast|_\infty$
					& 4.82E+5 & 3.43E+4 & 1.55E+2 & 6.14E-1 & 7.18E-3 \\
					convergence rate & & 1.90 & 3.89 & 3.99 & 3.21 \\
					\hline
					$|\delta_x^4\mathfrak{u}-{u^{(4)}}^\ast|_h$
					& 1.72E+7 & 8.41E+5 & 3.70E+3 & 1.46E+1 & 9.77E-1 \\
					convergence rate & & 2.18 & 3.91 & 3.99 & 1.95 \\
					$|\delta_x^4\mathfrak{u}-{u^{(4)}}^\ast|_\infty$
					& 5.34E+7 & 4.05E+6 & 1.77E+4 & 6.98E+1 & 6.59 \\
					convergence rate & & 1.86 & 3.92 & 3.99 & 1.70 \\
					\hline
				\end{tabularx}
				
				\caption{Accuracy errors in the approximations to $\{u^\ast, {u^{(1)}}^\ast, {u^{(2)}}^\ast, {u^{(3)}}^\ast, {u^{(4)}}^\ast \}$ by $\mathfrak{u}$ (Equation \eqref{utxxxxdiscgen}) and its discrete derivatives $\{ \widetilde{\delta_x} \mathfrak{u}, \widetilde{\delta_x^2} \mathfrak{u}, \delta_x^3 \mathfrak{u}, \delta_x^4 \mathfrak{u} \}$, for the test problem in Subsection \ref{ex2_section}. The convergence rates are calculated as per Equation \eqref{eq_convergence}. }
				
				\label{tab_example2_err}
				
			\end{table}
		}
		
	\clearpage
		
	\appendix
	\section{\textbf{TRUNCATION ERRORS}}
	\label{appendix_truncation}

		In this appendix we give a systematic derivation of the truncation errors for the discrete operators $\widetilde{\delta_x}, \widetilde{\delta_x^2}, \delta_x^3$ and $\delta_x^4$, as presented in Sections ~\ref{secdiscsetup} and ~\ref{secboundary}. In the first four Subsections ~\ref{trunc_1st_deriv} through ~\ref{trunc_4th_deriv} we establish optimal (fourth-order) truncation errors at interior points, and slower convergence on some near-boundary points, for all discrete derivatives. The truncation errors at the remaining (boundary and near-boundary points) are studied in subsections \ref{trunc_boundary} and \ref{trunc_3rd_deriv_near boundary}. These truncation errors are corroborated by the numerical results presented in Section \ref{secnumerresults}. As in Section ~\ref{secdiscsetup} we take $\Omega=[0,1].$
		
		\subsection{Hermitian derivative operator $\widetilde{\delta_x}$}
		\label{trunc_1st_deriv}
			
			We want to estimate the truncation error of the Hermitian derivative operator \eqref{eqoptilddelta}, i.e.:
			\begin{equation*}
				(\widetilde{\delta_x} u^\ast)_j - (u^{(1)})^\ast_j, \quad \forall j \in \{1,...,N-1\}.
			\end{equation*}
			
			Recall that the error on boundary points is identically zero.
			
			Equation (A.1) in \cite{Time_Evolution} specifies this error to fifth-order, without the actual derivation. We go here further up to the sixth-order term in order to prove 4th-order truncation errors of the various derivatives.
			
			For some sufficiently smooth $v : \Omega \rightarrow \mathbb{R}$ we use the Taylor expansion
			\begin{align} \label{eq_taylor}
				v^\ast_{j\pm1} =& v(x_j) \pm h v^{(1)}(x_j) + \frac{h^2}{2}v^{(2)}(x_j) \pm \frac{h^3}{6}v^{(3)}(x_j)\\
				&+ \frac{h^4}{24}v^{(4)}(x_j) \pm \frac{h^5}{120}v^{(5)}(x_j) + \frac{h^6}{720}v^{(6)}(x_j) + O(h^7), \nonumber \\ &\forall j \in \{1,...,N-1\}. \nonumber
			\end{align}
			
			Substituting \eqref{eq_taylor} into Equation \eqref{eqdefsigmasimpson} we get
			\begin{align} \label{eq_sigmax_taylor}
				(\sigma_x v^\ast)_j \equiv& \frac{v^\ast_{j-1} + 4 v^\ast_j + v^\ast_{j+1}}{6} \\ =& v(x_j) + \frac{h^2}{6}v^{(2)}(x_j) + \frac{h^4}{72}v^{(4)}(x_j) + \frac{h^6}{2160}v^{(6)}(x_j) + O(h^8), \nonumber \\ &\forall j \in \{1,...,N-1\}. \nonumber
			\end{align}
			
			Likewise, substituting the Taylor expansion into Equation \eqref{eq_delta_op} yields
			\begin{align} \label{eq_deltax_taylor}
				(\delta_x v^\ast)_j \equiv& \frac{v^\ast_{j+1}-v^\ast_{j-1}}{2h} \\ =& v^{(1)}(x_j) + \frac{h^2}{6}v^{(3)}(x_j) + \frac{h^4}{120}v^{(5)}(x_j) + \frac{h^6}{5040}v^{(7)}(x_j) + O(h^8), \nonumber \\ &\forall j \in \{1,...,N-1\} .\nonumber
			\end{align}
			
			Finally, by Equation ~\eqref{eq_delta2_op}
			\begin{align} \label{eq_deltax2_taylor}
				(\delta_x^2 v^\ast)_j &\equiv \frac{v^\ast_{j+1}-2v^\ast_j+v^\ast_{j-1}}{h^2} \\
				&= v^{(2)}(x_j) + \frac{h^2}{12} v^{(4)}(x_j) + \frac{h^4}{360} v^{(6)}(x_j) + O(h^6), \nonumber \\ &\forall j \in \{1,...,N-1\} .\nonumber
			\end{align}
			
			In the above equations we replace $v$ by $u$ and its derivatives. First, from \eqref{eq_sigmax_taylor} and \eqref{eq_deltax_taylor} we get
			\begin{align} \label{eq_her_der_error1}
				(\sigma_x u^{(1)\ast})_j - (\delta_x u^\ast)_j = & \frac{h^4}{180}u^{(5)}(x_j) + \frac{h^6}{3780}u^{(7)}(x_j) + O(h^8), \\ & \forall j \in \{1,...,N-1\}. \nonumber
			\end{align}

			By definition we have $\delta_x u^\ast = \sigma_x \widetilde{\delta_x} u^\ast$. Therefore Equation \eqref{eq_her_der_error1} can be written as
			\begin{equation} \label{eq_her_der_error2}
				\sigma_x (u^{(1)\ast} - \widetilde{\delta_x} u^\ast) = \frac{h^4}{180}u^{(5)\ast} + \frac{h^6}{3780}u^{(7)\ast} + O(h^8).
			\end{equation}
			
			Notice that $\sigma_x = I+\frac{h^2}{6} \delta_x^2$, where $I$ is the identity operator. We therefore know that the inverse $\sigma_x^{-1}$ exists, and we can express it with the Neumann series
			\begin{equation} \label{eq_inv_sigma_neumann}
				\sigma_x^{-1} = I - \frac{h^2}{6}\delta_x^2 + O(h^4).
			\end{equation}
			
			Now, Equation \eqref{eq_her_der_error2} can be rewritten as follows
			\begin{align*}
				\widetilde{\delta_x} u^\ast =& u^{(1)\ast} - \sigma_x^{-1}(\frac{h^4}{180}u^{(5)\ast} + \frac{h^6}{3780}u^{(7)\ast} + O(h^8)) \\
				=& u^{(1)\ast} - \frac{h^4}{180}u^{(5)\ast} - \frac{h^6}{3780}u^{(7)\ast} + \frac{h^2}{6}\delta_x^2(\frac{h^4}{180}u^{(5)\ast}) + O(h^8).
			\end{align*}
			
			Finally, using Equation \eqref{eq_deltax2_taylor} on the last term gives the following relation.
			\begin{equation} \label{eq_her_der_trunc_error}
				(\widetilde{\delta_x} u^\ast)_j = u_j^{(1)\ast} - \frac{h^4}{180}u_j^{(5)\ast} + \frac{h^6}{1512} u_j^{(7)\ast} + O(h^8), \quad \forall j\in\{1,...,N-1\}.
			\end{equation}
			
			Recall that $(\widetilde{\delta_x} u^\ast)_0 = u_0^{(1)\ast}$ and $(\widetilde{\delta_x} u^\ast)_N = u_N^{(1)\ast}$ are given.
			
		\subsection{Second-order discrete derivative operator $\widetilde{\delta_x^2}$}
		\label{trunc_2nd_deriv}
			
			For some $v : \Omega \rightarrow \mathbb{R}$ and for the interior points $\{x_2, ..., x_{N-2}\}$ we can write this derivative explicitly as a linear combination of the operators $\delta_x^2$ and $\delta_x \widetilde{\delta_x}$ (see \eqref{eqdefdelta2tild}).
			
			Substituting \eqref{eq_her_der_trunc_error} into Equation \eqref{eq_delta_op} entails \footnote{While we do not use the sixth-order term of $\widetilde{\delta_x} u^\ast$, we need to know that it does not become of fifth-order when we apply to it the $\delta_x$ operator, as is the case with higher-order terms on near-boundary points.}
			
			\begin{equation}\label{eqdelxtilddelx} \aligned
				(\delta_x \widetilde{\delta_x} u^\ast)_j =& \frac{{u^{(1)}_{j+1}}^\ast-{u^{(1)}_{j-1}}^\ast}{2h} - \frac{h^4}{180}\frac{{u^{(5)}_{j+1}}^\ast-{u^{(5)}_{j-1}}^\ast}{2h} + \frac{h^6}{1512}\frac{{u^{(7)}_{j+1}}^\ast-{u^{(7)}_{j-1}}^\ast}{2h} + O(h^7), \\ &\forall j\in\{2,...,N-2\}.
			\endaligned \end{equation}
			On near-boundary points we supplement Equation \eqref{eq_her_der_trunc_error} with the given boundary values $(\widetilde{\delta_x} u^\ast)_0 = u_0^{(1)\ast}$ and $(\widetilde{\delta_x} u^\ast)_N = u_N^{(1)\ast}$, so that
			\begin{align*}
				(\delta_x \widetilde{\delta_x} u^\ast)_1 =& \frac{u^{(1)\ast}_{2}-u^{(1)\ast}_{0}}{2h} - \frac{h^4}{180}\frac{u^{(5)\ast}_{2}}{2h} + \frac{h^6}{1512}\frac{u^{(7)\ast}_{2}}{2h} + O(h^7), \\
				(\delta_x \widetilde{\delta_x} u^\ast)_{N-1} =& \frac{u^{(1)\ast}_{N}-u^{(1)\ast}_{N-2}}{2h} - \frac{h^4}{180}\frac{u^{(5)\ast}_{N-2}}{2h} + \frac{h^6}{1512}\frac{u^{(7)\ast}_{N-2}}{2h} + O(h^7).
			\end{align*}
			
			Now, invoking Equation \eqref{eq_deltax_taylor} in Equation ~\eqref{eqdelxtilddelx} gives at interior points
			\begin{align} \label{eq_deltax_deltax_tilde_u}
				(\delta_x \widetilde{\delta_x} u^\ast)_j =& u^{(2)}(x_j) + \frac{h^2}{6} u^{(4)}(x_j) + \frac{h^4}{360}u^{(6)}(x_j) + O(h^6), \\ & \forall j\in\{2,...,N-2\}. \nonumber
			\end{align}
			
			Combining this result with \eqref{eq_deltax2_taylor} and \eqref{eqdefdelta2tild} we obtain
			
			\begin{equation}\label{eq_deltax2_tilde_trunc}
				(\widetilde{\delta_x^2} u^\ast)_j = u_j^{(2)}(x_j) + O(h^4), \quad \forall j\in\{2,...,N-2\}.
			\end{equation}
			
			At near boundary points we have a similar, but weaker result.
			\begin{align} \label{eq_deltax_deltax_tilde_u_near_boundary}
				(\delta_x \widetilde{\delta_x} u^\ast)_1 =& u^{(2)}(x_1) + \frac{h^2}{6} u^{(4)}(x_1) - \frac{h^3}{360}{u^{(5)}_{2}}^\ast + O(h^4), \\
				(\delta_x \widetilde{\delta_x} u^\ast)_{N-1} =& u^{(2)}(x_{N-1}) + \frac{h^2}{6} u^{(4)}(x_{N-1}) + \frac{h^3}{360}{u^{(5)}_{N-2}}^\ast + O(h^4), \nonumber
			\end{align}
			
			that leads to
			\begin{align}\label{eq_deltax2_tilde_trunc_near_boundary}
				(\widetilde{\delta_x^2} u^\ast)_1 &= u_j^{(2)}(x_1) - \frac{h^3}{360}{u^{(5)}_{2}}^\ast + O(h^4), \\
				(\widetilde{\delta_x^2} u^\ast)_{N-1} &= u_j^{(2)}(x_{N-1}) + \frac{h^3}{360}{u^{(5)}_{N-2}}^\ast + O(h^4) . \nonumber
			\end{align}
			
		\subsection{Third-order discrete derivative operator $\delta_x^3$}
		\label{trunc_3rd_deriv}
			We can write this operator (defined in \eqref{eqdef3rd}) as follows.
			\begin{align*}
				(\delta_x^3 u^\ast)_j &= ((2 \delta_x^2 \widetilde{\delta_x} - \delta_x \widetilde{\delta_x^2}) u^\ast)_j \\
				&= ((2 \delta_x^2 \widetilde{\delta_x} - 2 \delta_x \delta_x^2 + \delta_x \delta_x \widetilde{\delta_x}) u^\ast)_j \\
				&= (((2 \delta_x^2 + \delta_x \delta_x) \widetilde{\delta_x} - 2\delta_x \delta_x^2) u^\ast)_j, \quad \forall j\in\{2,...,N-2\}.
			\end{align*}
			
			 Using the Taylor expansion \eqref{eq_taylor} together with \eqref{eq_delta_op} and \eqref{eq_delta2_op}, the last term in the right-hand side of the equation can be evaluated as
			\begin{align}\label{eq_dx_dx2}
				(\delta_x \delta_x^2 u^\ast)_j &= \frac{u^\ast_{j+2}-2u^\ast_{j+1}+2u^\ast_{j-1}-u^\ast_{j-2}}{2h^3} = u^{(3)}(x_j) + \frac{h^2}{4} u^{(5)}(x_j) \\ &+ O(h^4), \quad \forall j\in\{2,...,N-2\}. \nonumber
			\end{align}
			
			If $v:\Omega\to \RR$ is a smooth function then
			\begin{align*}
				(\delta_x \delta_x v^\ast)_j =& \frac{v^\ast_{j+2}-2v^\ast_j+v^\ast_{j-2}}{4h^2}=v^{(2)}(x_j)+\frac{h^2}{3}v^{(4)}(x_j)+O(h^4), \\
				&\forall j\in\{2,...,N-2\}.
			\end{align*}
			
			The above equation, combined with Equation \eqref{eq_deltax2_taylor} yields
			\begin{align*}
				((2 \delta_x^2 + \delta_x \delta_x) v^\ast)_j =& \frac{v^\ast_{j-2} + 8 v^\ast_{j-1} - 18 v^\ast_j + 8 v^\ast_{j+1} + v^\ast_{j+2}}{4 h^2} \\
				=& 3 v^{(2)}(x_j) + \frac{h^2}{2} v^{(4)}(x_j) + O(h^4), \\
				&\forall j\in\{2,...,N-2\}.
			\end{align*}
			
			Let $v=u^{(1)}$ in the above equation. Invoking \eqref{eq_her_der_trunc_error} leads to
			\begin{align} \label{eq_2dx2+dx_dx}
				((2 \delta_x^2 + \delta_x \delta_x) \widetilde{\delta_x} u^\ast)_j =& 3 u^{(3)}(x_j) + \frac{h^2}{2} u^{(5)}(x_j) + O(h^4), \\
				&\forall j\in\{2,...,N-2\}. \nonumber
			\end{align}
			
			Finally, combining Equations \eqref{eq_dx_dx2} and \eqref{eq_2dx2+dx_dx} we get
			\begin{align} \label{eq_deltax3_trunc}
				(\delta_x^3 u^\ast)_j &= [3 u^{(3)}(x_j) + \frac{h^2}{2} u^{(5)}(x_j) + O(h^4)] \\
				&-2[u^{(3)}(x_j) + \frac{h^2}{4}u^{(5)}(x_j) + O(h^4)] \nonumber \\
				&= u^{(3)}(x_j) + O(h^4), \quad \forall j\in\{2,...,N-2\}. \nonumber
			\end{align}
	
			We are still missing the truncation errors of $\delta_x^3$ at the near-boundary points $j=1,\,N-1.$ They are derived below in Subsection ~\ref{trunc_3rd_deriv_near boundary}.
			
		\subsection{Fourth-order discrete derivative operator $\delta_x^4$}
		\label{trunc_4th_deriv}
		
			Recall the definition ~\eqref{eq_delta4_op} of the discrete biharmonic operator:
			\begin{equation*}
				(\delta_x^4 \fv)_j = \frac{12}{h^2}\Big[(\delta_x \widetilde{\delta_x} \fv)_j - (\delta_x^2 \fv)_j\Big], \quad \forall j\in\{1,...,N-1\}.
			\end{equation*}
			Let $u : \Omega \rightarrow \mathbb{R}$ be a smooth function. In view of
			 Equations ~\eqref{eq_deltax2_taylor} and ~\eqref{eq_deltax_deltax_tilde_u}, we get
			\begin{align} \label{eq_deltax4_trunc}
				(\delta_x^4 u^\ast)_j = \frac{12}{h^2}(&u^{(2)}(x_j) + \frac{h^2}{6} u^{(4)}(x_j) + \frac{h^4}{360}u^{(6)}(x_j) + O(h^6)\\
				-&u^{(2)}(x_j) - \frac{h^2}{12} u^{(4)}(x_j) - \frac{h^4}{360} u^{(6)}(x_j) + O(h^6)) \nonumber \\
				=&u^{(4)}(x_j) + O(h^4), \quad \forall j\in\{2,...,N-2\}. \nonumber
			\end{align}
			
			At near-boundary points we use Equation ~\eqref{eq_deltax_deltax_tilde_u_near_boundary} instead of ~\eqref{eq_deltax_deltax_tilde_u}. The expansion of $\delta_x \widetilde{\delta_x} u^\ast$ at $x_1, x_{N-1}$ contains a third-order term which does not cancel out, so instead of Equation \eqref{eq_deltax4_trunc} we have
			\begin{align} \label{eq_deltax4_trunc_near_boundary}
				(\delta_x^4 u^\ast)_1 =& u^{(4)}(x_1) + O(h), \\
				(\delta_x^4 u^\ast)_{N-1} =& u^{(4)}(x_{N-1}) + O(h) .\nonumber
			\end{align}
			
		\subsection{Truncation errors at the boundary}
		\label{trunc_boundary}
		
			We now examine the truncation error of the higher-order derivatives $u^{(2)}, u^{(3)}$ and $u^{(4)}$ at boundary points $x_0$ and $x_N$. We treat the boundary point $x=0,$ the treatment at $x=1$ being completely analogous.\\

			Recall that the system ~\eqref{eq_boundary_ode} - \eqref{eq_boundary_lincomb}, yields the approximate boundary values
			
			\begin{equation} \label{eq3bdryval}
				\fU_0 = \begin{pmatrix}
					{\mathfrak{u}^{(2)}_{0}} \\
					{\mathfrak{u}^{(3)}_{0}} \\
					{\mathfrak{u}^{(4)}_{0}}
				\end{pmatrix}.
			\end{equation}

			The system can be written in matrix form as
			\begin{equation}
				\alpha \fU_0 = \mathfrak{b},
			\end{equation}
			where
			\begin{equation*}
				\alpha = \begin{pmatrix}
					A_0^\ast & D_0^\ast & 1\\
					0 & \frac{1}{2h} & \frac{1}{6}\\
					\frac{2}{h^2} & \frac{1}{2h} & 0
				\end{pmatrix}, \quad
				\mathfrak{b} = \begin{pmatrix}
					f_0^\ast \\
					\frac{1}{2h} (\delta_x^3 \mathfrak{u})_2 - \frac{1}{6} (4 (\delta_x^4 \mathfrak{u})_1 + (\delta_x^4 \mathfrak{u})_2)\\
					(\delta_x^4 \mathfrak{u})_1 - \frac{2}{h^2} ((\widetilde{\delta_x^2} \mathfrak{u})_2 - 2(\widetilde{\delta_x^2} \mathfrak{u})_1) + \frac{1}{2h} (\delta_x^3 \mathfrak{u})_2
				\end{pmatrix}.
			\end{equation*}

			Now let
			\begin{equation} \label{eq3bdryvalexact}
				U_0^\ast = \begin{pmatrix}
					(u^{(2)}_0)^\ast \\
					(u^{(3)}_0)^\ast \\
					(u^{(4)}_0)^\ast
				\end{pmatrix}
			\end{equation}
			be the corresponding set of boundary values of the second, third and fourth derivatives of the exact solution.

			In the vector $\fb$ we replace the approximate solution $\fu$ by the grid values $u^\ast$ of the exact solution, thus getting the vector
			\begin{equation*}
				b^\ast = \begin{pmatrix}
					f_0^\ast \\
					\frac{1}{2h} (\delta_x^3 u^\ast)_2 - \frac{1}{6} (4 (\delta_x^4 u^\ast)_1 + (\delta_x^4 u^\ast)_2)\\
					(\delta_x^4 u^\ast)_1 - \frac{2}{h^2} ((\widetilde{\delta_x^2} u^\ast)_2 - 2(\widetilde{\delta_x^2} u^\ast)_1) + \frac{1}{2h} (\delta_x^3 u^\ast)_2
				\end{pmatrix}.
			\end{equation*}
			
			Let $V$ be the solution of
			\begin{equation} \label{eqVbast}
				\alpha V=b^\ast.
			\end{equation}
			The required truncation error is therefore
			\begin{equation}
				\fe= U_0^\ast-V.
			\end{equation}
			Replacing in $b^\ast$ the finite difference operators by exact derivatives, we obtain the vector
			\begin{equation}
				B^\ast=\begin{pmatrix}
					f_0^\ast \\
					\frac{1}{2h} (u^{(3)}_{2})^\ast - \frac{1}{6} \Big[4 (u^{(4)}_{1})^\ast + (u^{(4)}_{2})^\ast\Big]\\ \\
					(u^{(4)}_{1})^\ast - \frac{2}{h^2} \Big[{u^{(2)}_{2}}^\ast - 2(u^{(2)}_{1})^\ast\Big] + \frac{1}{2h} (u^{(3)}_{2})^\ast
				\end{pmatrix}.
			\end{equation}

			Notice that $\alpha$ depends only on exact values of the coefficients. Thus $U_0^\ast$ satisfies a system of the form
			\begin{equation}\label{eqU0Bast}
				\alpha U_0^\ast = B^\ast+\fr^\ast,
			\end{equation}
			where in view of the estimates ~\eqref{eq_boundary_simpson_continuous} and ~\eqref{eq_boundary_lincomb_continuous},
			\begin{equation}
				\fr^\ast = \begin{pmatrix}
					0 \\ O(h^4) \\ O(h^4)
				\end{pmatrix}
			\end{equation}
			Invoking the truncation errors obtained in Subsections \ref{trunc_2nd_deriv} - \ref{trunc_4th_deriv} above we get,
			\begin{equation*}
				B^\ast-b^\ast=\fs,
			\end{equation*}
			where
			\begin{equation*}
				\fs=\begin{pmatrix}
					0 \\ O(h) \\ O(h)
				\end{pmatrix}.
			\end{equation*}
			
			Thus, from ~\eqref{eqU0Bast} and the above estimates we have
			\begin{equation}
				\alpha U_0^\ast = b^\ast+\fw
			\end{equation}
			where
			\begin{equation*}
				\fw=\fr^\ast+\fs=\begin{pmatrix}
					0 \\ O(h) \\ O(h)
				\end{pmatrix}.
			\end{equation*}

			By the definition of the error $\fe$ we have
			\begin{equation*}
				\fe=\alpha^{-1}\fw.
			\end{equation*}

			To estimate the error it remains to evaluate the inverse of $\alpha.$
			\begin{gather*}
				\alpha^{-1} = \frac{1}{12 - 4 D_0^\ast h + A_0^\ast h^2} \begin{pmatrix}
					h^2 & -6 h^2 & 6 h^2 - D_0^\ast h^3\\
					-4h & 24 h & 2A_0^\ast h^3\\
					12 & -24D_0^\ast h + 6 A_0^\ast h^2 & -6 A_0^\ast h^2
				\end{pmatrix} \\
				= \begin{pmatrix}
					O(h^2) & O(h^2) & O(h^2)\\
					O(h) & O(h) & O(h^3)\\
					O(1) & O(h) & O(h^2)
				\end{pmatrix}.
			\end{gather*}
			
			Thus
			\begin{equation} \label{boundary_values_truncation_error}
				\fe = \begin{pmatrix}
					O(h^3) \\ O(h^2) \\ O(h^2)
				\end{pmatrix}.
			\end{equation}
			
			\textbf{Summary.} At the boundary points $j=0,\,N$ the calculated values of the derivatives are subject to truncation errors as follows.
		
			\begin{itemize}
				\item The truncation error for the second-order discrete derivative $(\widetilde{\delta_x^2} u^\ast)_{0,N}$ is $O(h^3).$
				
				\item The truncation error for the third-order discrete derivative $(\delta_x^3 u^\ast)_{0,N}$ is $O(h^2).$

				\item The truncation error for the fourth-order discrete derivative $(\delta_x^4 u^\ast)_{0,N}$ is $O(h^2).$
			\end{itemize}

			The numerical results presented in Section ~\ref{secnumerresults} agree with the analytical bounds derived here.

			
		\subsection{Truncation errors of the third-order discrete derivative at near-boundary points}
			\label{trunc_3rd_deriv_near boundary}

			In Subsection ~\ref{trunc_3rd_deriv} we established the optimal $O(h^4)$ truncation errors for $\delta_x^3$ at interior points $j\in\set{2,3,\ldots,N-2}.$ In addition, in Subsection ~\ref{trunc_boundary} we obtained the truncation rate of $O(h^2)$ for the boundary values ($j=0,\,N$). We now conclude the study of the truncation errors by considering the remaining near-boundary points ($j=1,\, N-1$).

			Note that in view of ~\eqref{eq_deltax2_tilde_trunc_near_boundary} and ~\eqref{boundary_values_truncation_error} the truncation error for second-order discrete derivative is $O(h^3)$ at boundary and near-boundary points.
			\begin{equation}\label{eq_deltax2_tilde_trunc_full}
				(\widetilde{\delta_x^2} u^\ast)_j = u_j^{(2)}(x_j) + O(h^3), \quad j=0,1,\,N-1,N.
			\end{equation}
			Invoking Equation ~\eqref{eq_deltax_taylor} yields
			\begin{equation*}\label{eq_3rd_deriv_approx1}
				(\delta_x \widetilde{\delta_x^2} u^\ast)_j = u^{(3)}(x_j) + O(h^2), \quad j=1,N-1.
			\end{equation*}
			Also by Equation ~\eqref{eq_deltax2_taylor}, Equation ~\eqref{eq_her_der_trunc_error} and the boundary conditions $u^{(1)}(x_0)= u^{(1)}(x_N)=0$
			\begin{equation*}\label{eq_3rd_deriv_approx2}
				(\delta_x^2 \widetilde{\delta_x} u^\ast)_j = u^{(3)}(x_j) + O(h^2), \quad j=1,\,N-1.
			\end{equation*}
			As mentioned above, the truncation error is optimal $O(h^4)$ at interior points $j\in\set{2,3,\ldots,N-2}.$
			However, the last two estimates imply, by the definition of $\delta_x^3$ ~\eqref{eq_delta3_op}, that
			\begin{equation}\label{eq_deltax3_trunc_full}
				(\delta_x^3u^\ast)_j= u^{(3)}(x_j) + O(h^2), \quad j=1,\,N-1.
			\end{equation}
			Thus its truncation error is $O(h^2)$ at near-boundary points $j=1,\,N-1$, like the truncation error at the boundary.

	\bibliographystyle{abbrv}

	\addresseshere

\end{document}